\documentclass[a4paper, 11pt,psamsfonts]{amsart}
\usepackage{amsmath}
\usepackage{amsthm}
\usepackage{amssymb}
\usepackage{amscd}
\usepackage{amsfonts}
\usepackage{amsbsy}
\usepackage{enumerate}
\usepackage{graphicx}
\usepackage{calc}
\usepackage{xcolor}
\usepackage{setspace}

\newtheorem {theorem} {Theorem}
\newtheorem {proposition}{Proposition}
\newtheorem {corollary}{Corollary}
\newtheorem {lemma}{Lemma}
\newtheorem {definition} {Definition}
\newtheorem {remark}{Remark}
\newtheorem {example} {Example}

\title[Centralizers and normalizers]
{Centralizers and normalizers of local analytic and formal vector fields}

\author[N. Kruff, S. Walcher, X. Zhang]
{Niclas Kruff$^1$, Sebastian Walcher$^1$, Xiang Zhang$^2$ }

\address{$^1$ Lehrstuhl A f\"ur Mathematik, RWTH
Aachen, 52056 Aachen, Germany} \email{\{niclas.kruff,walcher\}@matha.rwth-aachen.de}

\address{$^2$ School of Mathematical Sciences and MOE-LSC, Shanghai Jiao Tong
University, Shanghai, 200240, P. R. China}
\email{xzhang@sjtu.edu.cn}

\subjclass[2010]{34A34; 34C14; 37G05; 37G40.}

\keywords{Local vector field;  centralizer; normalizer; normal form; Jacobi multiplier.}
\begin{document}

\begin{abstract} We investigate the structure of the centralizer and the normalizer of a local analytic or formal differential system at a nondegenerate stationary point, using the theory of Poincar\'e-Dulac normal forms. Our main results are concerned with the formal case. We obtain a description of the relation between centralizer and normalizer, sharp dimension estimates when the centralizer of the linearization has finite dimension, and lower estimates for the dimension of the centralizer in general. For a distinguished class of linear vector fields (which is sufficiently large to be of interest) we obtain a precise characterization of the centralizer for corresponding normal forms in the generic case. {Moreover, in view of their relation to normalizers, we discuss inverse Jacobi multipliers and obtain existence criteria and nonexistence results for several classes of vector fields}.

\end{abstract}

\maketitle

\section{Introduction}
\subsection{Background and motivation} Symmetries of differential equations were first systematically investigated and applied  by Sophus Lie; see \cite{Li 1} and subsequent works; for more contemporary accounts we refer to Olver \cite{Ol 1}, Stephani \cite{Ste}, among many others. The existence of nontrivial symmetries for an ordinary differential equation has consequences for the structure of its solutions, allowing reduction and thus facilitating the analysis. Moreover symmetries are relevant for the dynamics of the system and for integrability questions; see {e.\ g.}\ the monograph \cite{Zh17}, and {the research papers} Aziz et al.\ \cite{ALP14}, Colak et al.\ \cite{CLV14},
Freire et al.\ \cite{FGG}, Garcia et al.\ \cite{GMS},  Gin\'e et al \ \cite{GGL}, Llibre and Valls \cite{LV11}. While there exist algorithms to determine symmetries of higher order ordinary differential equations, the class of (autonomous) first order equations does not allow such an approach (and is therefore even called exceptional by Stephani \cite{Ste}). The present paper is devoted to this class.

Following Lie's fundamental approach, we are mainly interested in local one-parameter groups of symmetries or orbital symmetries of an ordinary differential equation $\dot x=f(x)$. As is well known, the infinitesimal generator $g$ of such a one-parameter group commutes with $f$ (thus $[g,\,f]=0$), respectively normalizes $f$ (thus $[g,\,f]=\lambda f$ with some scalar-valued function  $\lambda$), {see e.g.\ \cite{WaMul}.}
While frequently one investigates vector fields with a priori known symmetries, we consider the direct problem here: Given an analytic vector field $f$, find conditions (and restrictions) that possible infinitesimal symmetries or orbital symmetries have to satisfy.
Local obstructions to the existence of nontrivial commuting or normalizing vector fields appear at stationary points, and only at stationary points. We will discuss the centralizer resp.\ the normalizer at those stationary points which are nondegenerate in the sense that the linearization of the vector field is not nilpotent. We will consider complex vector fields in the present paper; the transfer and application to the real case is unproblematic. We note here that a more general study of the local analytic case in dimension two is due to Cerveau and Lins Neto \cite{CLN}. {Local results have obvious implications for the structure of the global centralizer or normalizer of a vector field on an open and connected set, since restriction from the global to the local setting induces an injective morphism of the corresponding structures.} For the local study, we also turn to formal power series and vector fields, which allows to employ properties of Poincar\'e-Dulac normal forms.
\subsection{Overview and main results}
The paper is organized as follows. In Section \ref{basicsec} we recall some notions and facts; in particular we recall that the structures of local centralizer and normalizer are well understood near nonstationary points (Proposition \ref{nonstatprop}).  For stationary points we turn to normal form theory, and we characterize the relation between formal centralizer and normalizer for formal vector fields in PDNF in Theorem \ref{linnormthm}. \\
In Section \ref{findimsec} we discuss the case when $f$ is in PDNF and the semisimple part of $Df(0)$ has finite dimensional centralizer, which includes the case that $f$ is necessarily linear. Here we clarify the relation between centralizer and normalizer (Proposition \ref{findimcentnorm}) and obtain general sharp estimates for the (necessarily finite) dimension of the centralizer in Theorems \ref{finitedimT1} and \ref{finitedimT2}, illustrated by a number of examples. In particular {we prove that} the dimension of the formal centralizer is always greater or equal to the space dimension. We also note some consequences for the local analytic case. \\
In Section \ref{infidimsec} we discuss vector fields in PDNF with infinite dimensional centralizer of $Df(0)$. We first recall some tools (including a symmetry reduction) from normal form theory and apply these to investigate the centralizer, obtaining a lower estimate for its dimension in Theorem \ref{lowerest}. Since some underlying algebraic structures (the algebra of polynomial first integrals and the module of polynomial vector fields commuting with a given semisimple linear map) may be rather complicated, a complete investigation seems currently out of reach. But for the distinguished case when the underlying algebraic structures are as simple as possible we obtain in Theorems \ref{quadcentthm} and \ref{disalgcentthm} a precise description of the centralizer for (in a precise sense) generic vector fields. Moreover we exhibit several classes of examples for this distinguished case, including a class of coupled oscillators, and {we give} a description of all such vector fields in dimension three.\\
We begin Section \ref{jacobisec} by recalling the relation between normalizer elements and Jacobi last multipliers, and we establish in Proposition \ref{divnofoprop} some special properties of formal inverse Jacobi multipliers for PDNF. We proceed to investigate formal Jacobi last multipliers for several examples, including a subset of the distinguished class from Section \ref{infidimsec}. { This includes all three dimensional vector fields in the above class.
For those we provide conditions for possible mulitipliers and apply these to show nonexistence of formal inverse Jacobi multipliers for a generic case.} {We note that some of the results on three dimensional vector fields originate from the doctoral dissertation \cite{Kruffdiss}.}

%%%%%%%%%%%%%%%%%%%%%%%%%%%%%%%%%%%%%%%%%%%%%%%%%%%%%%%
\section{Basic notions and results}\label{basicsec}

\subsection{Definitions and known facts}

We let $\emptyset\not=U^*\subseteq U\subseteq \mathbb C^n$, both subsets open and connected. Given an ordinary differential equation
\begin{equation}\label{deq}
\dot x = f(x) \quad\text{ on  } U
\end{equation}
with analytic right hand side, we consider the centralizer
\begin{equation}\label{defcent}
{\mathcal C}_{U^*}(f):=\left\{ g;\, g\text{  analytic on  }U^*,\,\left[g,\,f\right]=0\right\},
\end{equation}
where the Lie bracket is defined by
\[
\left[g,\,h\right](x):=Dh(x)g(x)-Dg(x)h(x),
\]
as usual, with $Dh(x)$ denoting the Jacobian matrix of the vector field $h(x)$. Furthermore we consider the normalizer
\begin{equation}\label{defnorm}
{\mathcal N}_{U^*}(f):=\left\{ g;\, g\text{  analytic on }U^*,\,\left[g,\,f\right]=\mu f\text{  for some analytic  }\mu\right\}.
\end{equation}
As is well known from the work of Lie \cite{Li 1} (see also more contemporary accounts in Olver \cite{Ol 1}, Stephani \cite{Ste}, and in \cite{WaMul}), the local flow of $g\in{\mathcal C}_{U^*}(f)$ yields a local one-parameter symmetry group for system \eqref{deq}, while the local flow of $g\in{\mathcal N}_{U^*}(f)$ yields a local one parameter group of orbital symmetries of \eqref{deq}, thus preserving solution trajectories but not necessarily their parameterization.

We recall some further notions. For analytic $\phi:\,U^*\to \mathbb C$ and $g:\,U^*\to\mathbb C^n$ we define the Lie derivative $X_g(\phi)$ by
\begin{equation}\label{deflied}
 X_g(\phi)(x):=D\phi(x)g(x);
\end{equation}
and $\phi$ is by definition a first integral of $g$ if $X_g(\phi)=0$. {(Note that we include constant functions as first integrals; this is more appropriate from an algebraic perspective.)}
Moreover we have the identities
\begin{equation}\label{bracket}
X_gX_h-X_hX_g=X_{[g,\,h]},
\end{equation}
as well as
\begin{equation}\label{mixid}
\left[g,\,\psi\cdot h\right]=X_g(\psi)\cdot h+\psi\cdot \left[g,\,h\right]
\end{equation}
for analytic functions $\psi$ and analytic vector fields $g$, $h$.

An element of the centralizer ${\mathcal C}_{U^*}(f)$ may be seen as belonging two vector spaces, which give rise to two notions of dimension. (For background and more facts see e.g.\ Bianchi \cite{Bi} and Hermann \cite{Her}.) On the one hand, $g\in {\mathcal C}_{U^*}(f)$ may be seen as an element of the vector space $\mathbb L^n$, where $\mathbb L$ denotes the field of meromorphic functions on $U^*$, i.e., functions that can locally be expresssed as quotients of analytic functions. Linear dependence over this field means linear dependence in $\mathbb C^n$ of the function values at every point, as the next result shows.
\begin{lemma}%{
Let $h_1,\ldots,h_r$ be analytic vector fields on $U^*$. Then $h_1,\ldots,h_r$ are linearly dependent over $\mathbb L$ if and only if  $h_1(y),\ldots,h_r(y)\in \mathbb C^n$ are linearly dependent over $\mathbb C$, for every $y\in U^*$.
%}
\end{lemma}
\begin{proof}
Let $H\in\mathbb L^{n\times r}$ denote the matrix with columns $h_1,\ldots, h_r$, and let $\Delta\in\mathbb L$ denote any of its $r\times r$ minors. Then $\Delta$ is not the zero function if and only if $\Delta(y)\not=0$ for some $y\in U^*$.
\end{proof}
On the other hand we may consider ${\mathcal C}_{U^*}(f)$ as a vector space (and a Lie algebra) of functions over $\mathbb C$. We will denote by $\dim {\mathcal C}_{U^*}(f)$  the dimension of this $\mathbb C$-vector space, which we are mainly interested in. We include a proof of the following fact for the reader's convenience.

\begin{proposition}\label{meroprop}
Let $h_1,\ldots,h_r\in {\mathcal C}_{U^*}(f)$ be linearly independent over $\mathbb C$ but linearly dependent over $\mathbb L$. Then the differential equation \eqref{deq} admits a nonconstant meromorphic first integral on $U^*$. In particular $f$ admits a meromorphic first integral whenever $\dim  {\mathcal C}_{U^*}(f)>n$.
\end{proposition}
\begin{proof}%{
With no loss of generality we may assume that, for some $s<r$, the vector fields $h_1,\ldots,h_s$ are linearly independent but $h_1,\ldots,h_{s+1}$ are linearly dependent over $\mathbb L$, hence
\[
h_{s+1}=\mu_1h_1+\cdots+\mu_sh_s
\]
with suitable $\mu_i\in\mathbb L$. Using \eqref{mixid} we find
\[
0=\left[f,h_{s+1}\right]=\sum_{i=1}^s\mu_i\left[f,\,h_i\right]+ \sum_{i=1}^s X_f(\mu_i)h_i=\sum_{i=1}^s X_f(\mu_i)h_i
\]
and therefore $X_f(\mu_1)=\cdots =X_f(\mu_s)=0$, whence every $\mu_i$ is a meromorphic first integral (possibly constant). Due to the linear independence of the $h_i$ over $\mathbb C$, not all $\mu_i$ are constant.

The last assertion follows since there cannot be more than $n$ linearly independent elements in $\mathbb L^n$.
%}
\end{proof}

Given any $y\in U^*\subseteq U$, we may pass from $ {\mathcal C}_{U^*}(f)$ to the local analytic centralizer
\begin{equation}\label{defloccent}
{\mathcal C}_y(f):=\left\{ g;\, g\text{  analytic at }y,\,\left[g,\,f\right]=0\right\},
\end{equation}
and similarly to  the local analytic normalizer ${\mathcal N}_y(f)$. We call $\mathbb F$ the field of local meromorphic functions in $y$, and note that Proposition \ref{meroprop} carries over to the local setting.
The structure of the local centralizer and of the local normalizer is well understood near any nonstationary point; we recall the pertinent result for convenient reference.
\begin{proposition}\label{nonstatprop}
Let $n>1$ and $y\in U$ nonstationary for system \eqref{deq}. Then:
\begin{enumerate}[$(a)$]
\item There exist $g_1,\ldots, g_n\in {\mathcal C}_y(f)$ that are linearly independent over $\mathbb F$ such that
\[
{\mathcal C}_y(f)=\left\{\sum\mu_i g_i;\,X_f(\mu_1)=\cdots =X_f(\mu_n)=0\right\}.
\]
In particular ${\mathcal C}_y(f)$ has infinite dimension over $\mathbb C$.
\item
The normalizer has the following structure:
\[
{\mathcal N}_y(f)={\mathcal C}_y(f)+\left\{\beta f;\,\beta\text{  analytic in }y\right\}.
\]
\item Given $g\in{\mathcal N}_y(f)$ and $g(y)\not=0$, there exists a local analytic $\sigma$ with $\sigma(y)=1$ such that
\[
\left[g,\,\sigma f\right]=0;
\]
thus $g$ lies in the centralizer of some vector field with the same local solution trajectories as $f$.
\end{enumerate}
\end{proposition}
\begin{proof}
By the straightening theorem one may assume that $f$ is a constant vector field, say $f=e_1$, where $e_1$ is the first element of the standard basis of $\mathbb C^n$. (Recall that Lie brackets are compatible with coordinate transformations.) Then $\left[g,\,f\right]=0$ is equivalent to $\partial g/\partial x_1=0$, thus the entries of $g$ depend locally only on $x_2,\ldots,x_n$. Therefore we may take $g_i=e_i$ as constants, and the coefficient functions are first integrals of $f$. To prove part $(b)$, assume $[g,\,f]=\lambda f$. Then
\[
[g-\alpha f,\,f]=(\lambda-X_f(\alpha))\,f
\]
for any $\alpha$, and the (linear first order partial differential) equation $\lambda=X_f(\alpha)=\partial \alpha/\partial x_1$ has a local solution $\beta$, hence $g-\beta f\in {\mathcal C}_y(f)$. Finally, for part (c) note that $\left[g,\,f\right]=\lambda f$ implies $\left[g,\,\sigma f\right]=(X_g(\sigma)+\lambda\sigma)\,f$. Straightening $g$ (or invoking familiar facts about linear partial differential equations), one sees that the equation $X_g(\sigma)+\lambda\sigma=0$ has a solution with $\sigma(y)=1$.
\end{proof}
{The proposition and its proof show that the structures of the local centralizer or the local normalizer reflect characteristic properties of a vector field, and thus is of interest, only near stationary points. These structures are therefore in the focus of the present paper.}
%%%%%%%%%%%%%%%%%%%%%%%%%%%%%%%%%%%%%%%%%%%%%%%%%
\subsection{Local centralizers and normalizers near stationary points}
We consider the centralizer of the analytic vector field $f$ near a stationary point, which we take to be $0$. Thus $f$ admits a Taylor expansion
\begin{equation}\label{taylor}
f(x)=Ax+\sum_{j\geq 2} f_j(x)
\end{equation}
with $A\in \mathbb C^{n\times n}$, and each $f_j$ a homogeneous vector polynomial of degree $j$.  In order to enable working with Poincar\'e-Dulac normal forms, we extend the discussion from power series with a nonempty domain of convergence to formal power series. We denote by $\mathbb C[[x_1,\ldots, x_n]]$ the algebra of formal power series, and by $\mathbb C((x_1,\ldots,x_n))$ its quotient field, noting that the definitions and identities such as \eqref{defcent}, \eqref{deflied}, \eqref{bracket} and \eqref{mixid} carry over. Moreover we consider
\[ %begin{equation}\label{defforcent}
{\mathcal C}_0^{\rm for}(f):=\left\{g\in\mathbb C[[x_1,\ldots,x_n]]^n;\,\left[g,\,f\right]=0\right\},
\] %end{equation}
and analogously ${\mathcal N}_0^{\rm for}(f)$, whose elements satisfy $[g,\,f]=\lambda f$ with some $\lambda\in \mathbb C[[x_1,\ldots,x_n]]$. \\
Poincar\'e-Dulac normal forms (PDNF) and their special properties make many arguments more transparent. We recall the coordinate-invariant definition (see e.g.\ \cite{Wa91}):
Given the Jordan-Chevalley decomposition $A=A_s+A_n$ with $A_s$ semisimple, $A_n$ nilpotent and $\left[A_s,\,A_n\right]=0$, $f$ as given in \eqref{taylor} is in PDNF if $[A_s,\,f]=0$; equivalently each $[A_s,\,f_j]=0$. As is well known (see e.g. \cite{Br89}), any local analytic vector field admits a formal power series transformation to PDNF, but in general no convergent transformation exists. Obviously when $f$ is analytic then ${\mathcal C}_0(f)\subseteq {\mathcal C}_0^{\rm for}(f)$ and ${\mathcal N}_0(f)\subseteq {\mathcal N}_0^{\rm for}(f)$; this allows to transfer some results from the formal to the analytic case. \\
We will frequently use the following facts (see e.g.\ \cite{Wa91}, Propositions 1.3 and 1.4):
\begin{lemma}\label{jochlem} Given the Jordan-Chevalley decomposition $A=A_s+A_n$ into semisimple and nilpotent part, the following hold:
\begin{enumerate}[$(a)$]
\item The Lie derivative $X_A$ sends every space $S_k$ of homogeneous polynomials of degree $k$ to itself, and $X_A=X_{A_s}+X_{A_n}$ is the corresponding Jordan-Chevalley decomposition.
\item The adjoint action ${\rm ad}\,A$ sends every space ${\mathcal P}_k$ of homogeneous polynomial vector fields of degree $k$ to itself, and ${\rm ad}\,A={\rm ad}\,A_s+{\rm ad}\,A_n$ is the corresponding Jordan-Chevalley decomposition.
\end{enumerate}
\end{lemma}
A starting point for the investigation of the formal centralizer is the following well-known result, which is a consequence of Proposition 1.5 in \cite{Wa91}. We include a direct proof here, for the reader's convenience.
\begin{proposition}\label{lincentprop}
Let system \eqref{deq} be in Poincar\'e-Dulac normal form, and $g\in{\mathcal C}_0^{\rm for}(f)$. Then also $g\in{\mathcal C}_0^{\rm for}(A_s)$.
\end{proposition}
\begin{proof}
 Let  $g=g_r+\cdots$, with $g_r\neq 0$. Then $[g,f]=0$ is equivalent to
\begin{equation}\label{centcondeq}
 [A,g_{r+j}]+[f_2,g_{r+j-1}]+\cdots +[f_{j+1},g_r]=0, \qquad \text{all}\;\,j\geq 0.
\end{equation}
We show by induction that $[A_s,g_{r+j}]=0$ for all $j\geq 0$. For $j=0$ the assertion follows from $[A,g_r]=0$ and $\ker({\rm ad}\,A)\subseteq \ker({\rm ad}\,A_s)$ on $\mathcal{P}_r$. For the induction step apply ${\rm ad}\, A_s$ to \eqref{centcondeq} and recall that $A_s$ and $f_j$ commute for all $j$. Therefore
\[
 \bigl[A,[A_s,g_{r+j}]\bigr]+\bigl[f_2,[A_s,g_{r+j-1}]\bigr]+\cdots +\bigl[f_{j+1},[A_s,g_r]\bigr]=0.
\]
By induction hypothesis all terms after the first one vanish, hence one sees that $g_{r+j}\in \ker({\rm ad}\,A_s)^2=\ker{\rm ad}\,A_s$.
\end{proof}
We next turn to the formal normalizer. The connection to the centralizer is not as straightforward as for nonstationary points in Proposition \ref{nonstatprop}, but some properties persist.
\begin{theorem}\label{linnormthm}
Let system \eqref{deq} be in Poincar\'e-Dulac normal form, with $A_s\not=0$, and $g\in{\mathcal N}_0^{\rm for}(f)$ with $g(0)=0$.
\begin{enumerate}[$(a)$]
\item There exists a power series $\beta$ such that
\[
\left[g-\beta f,\,f\right]=\alpha f,\quad X_{A_s}(\alpha)=0
\]
with some $\alpha=\alpha_0+\alpha_1+\cdots\in\mathbb C[[x_1,\ldots,x_n]]$, and $\alpha_0=0$.\\
In particular
\[
\mathcal{N}_0^{\rm for}(f)=\mathcal{C}_0^{\rm for}(f)+\left\{\beta f;\,\beta\in\mathbb C[[x_1,\ldots,x_n]]\right\}
\]
whenever $A_s$ admits only constant formal first integrals.
\item The vector field $h:=g-\beta f$ satisfies $[A_s,h]=0$.
\item There exists an invertible series $\sigma=1+\sigma_1+\cdots$ and $\alpha^*\in\mathbb C[[x_1,\ldots,x_n]]$ such that $X_{A_s}(\sigma)=0$ and
\[
\left[h,\,\sigma f\right]=\alpha^* \cdot \sigma f,\quad  X_{A_s}(\alpha)^*= X_{C_s}(\alpha^*)=0.
\]
\end{enumerate}

\end{theorem}
\begin{proof} Let $[g,\,f]=\lambda f$, with $g=g_r+\cdots$, $g_r\not=0$, and therefore $\lambda=\lambda_{r-1}+\cdots$.
\begin{enumerate}[$(i)$]
\item We first prove that $A_s\not=0$ implies $\lambda_0=0$. This is trivial when $r>1$. In case $r=1$, hence $g_1=C$ linear, one has $[C,A]=\lambda_0 A$, which implies $[[C,\,A],\,A_s]=0$ and therefore $[C,\,A_s]=0$  by semisimplicity of ${\rm ad}\,A_s|_{\mathcal P_1}$. We may assume that
\[
A_s={\rm diag}\,(\theta_1 I_{n_1},\ldots,\theta_r I_{n_r})
\]
is in block diagonal form, with pairwise different $\theta_j$, and $I_{n_j}$ denoting the $n_j\times n_j$ identity matrix.
Then by Lemma \ref{lincommutelem} in the Appendix one has
\[
A_n={\rm diag}(N_1,\ldots,N_r),\quad C={\rm diag}(C_1,\ldots,C_r)
\]
with nilpotent matrices $N_j$ and matrices $C_j$ of size $n_j\times n_j$.
Now the relation $[C,A]=\lambda_0A$ implies
\[
N_jC_j-C_jN_j=\lambda_0(\theta_jI_{n_j}+N_j),\quad i\leq j\leq r.
\]
The left hand side has trace zero, therefore $\lambda_0\theta_j=0$ for all $j$. From $A_s\not=0$ one now finds $\lambda_0=0$. Application to $g-\beta f$ for any series $\beta$ shows that $\alpha_0=0$ in part (a).
\item {We turn to prove part $(a)$. In case $r=1$ we have already seen that $[g_1,\,A]=0$. For $r>1$ we first show that there exists $\beta_{r-1}\in S_{r-1}$ such that
\[
[g_r-\beta_{r-1}A,\,A]=\alpha_{r-1}A
\]
with some $\alpha_{r-1}$ in the kernel of $X_{A_s}$.
At degree $r$ the normalizer condition yields}
\[
[g_r,A]=\lambda_{r-1} A,
\]
and thus for every homogeneous $\beta_{r-1}$ of degree $r-1$ one gets
\[
[g_r-\beta_{r-1}A,\,A]=(\lambda_{r-1}+X_A(\beta_{r-1}))A.
\]
Since $S_k$ is the sum of the image of $X_A$ and the kernel of $X_{A_s}$ (as a consequence of Lemma \ref{jochlem}), one can choose $\beta_{r-1}$ such that $\alpha_{r-1}:=\lambda_{r-1}+X_A(\beta_{r-1})\in \ker(X_{A_s})$.\\
We proceed by induction on $j$ to show: If $[g,\,f]=\lambda f$, and $X_{A_s}(\lambda_{r-1})=\cdots=X_{A_s}(\lambda_{r-2+j})=0$ then there exists a homogeneous polynomial $\beta_{r-1+j}$ such that
\[
[g-\beta_{r-1+j} f,\,f]=\alpha f \text{  and  } X_{A_s}(\alpha_{r-1})=\cdots =X_{A_s}(\alpha_{r-1+j})=0.
\]
Clearly one has $\alpha_i=\lambda_i$ for $r-1\leq i\leq r-2+j$. Evaluating the normalizer condition at degree $r+j$, one finds
\[
\begin{array}{ll}
&[g_{r+j},A]+[g_{r+j-1},f_2]+\cdots+[g_r,f_{j+1}]\\
&\quad= (\lambda_{r-1+j}+X_{A_s}(\beta_{r-1+j}))A+\alpha_{r-2+j}f_2+\cdots+\alpha_{r-1}f_j,
\end{array}
\]
and as before one may choose $\beta_{r-1+j}$ so that $\lambda_{r-1+j}+X_A(\beta_{r-1+j})\in \ker(X_{A_s})$.
{Now let ${\bf m}:=\left<x_1,\ldots,x_n\right>$ be the maximal ideal of $\mathbb C[[x_1,\ldots,x_n]]$. Since any sequence
\[
g-\sum_{j=0}^k\beta_{r-1+j}f
\]
with homogeneous $\beta_i$ of degree $i$ converges to some formal power series in the ${\bf m}$-adic topology, part $(a)$ is proven.}
\item We turn to part $(b)$. From $[h,\,f]=\alpha f$ and $X_{A_s}(\alpha)=0$ one obtains
\[
\begin{array}{rcl}
0=[A_s,\,\alpha f]&=&[A_s,\,[h,\,f]] \\
                   &=& -[f,\,[A_s,\,h]-[h,\,[f,\,A_s]]\\
                    &=&[[A_s,\,h],\,f]
\end{array}
\]
with the Jacobi identity and $[A_s,\,f]=0$. Now Proposition \ref{lincentprop} shows that $[A_s,\,[A_s,\,h]]=0$, which implies $[A_s,\,h]=0$ with Lemma \ref{jochlem}.
\item For the proof of part $(c)$, we let $h=C+h_2+\cdots$, make the ansatz $\sigma=1+\sum_{j>0}\sigma_j$, with all $X_{A_s}(\sigma_j)=0$, and evaluate
\[
[h,\,\sigma f]=(X_h(\sigma)+\alpha \sigma)f
\]
degree by degree. Recall that $C$ and $A_s$ commute by part $(i)$, hence $X_C$ maps each $S_k^*:=\ker X_{A_s}|_{S_k}$ to itself, and $S_k^*=\ker X_{C_s}|_{S_k^*}\oplus{\rm im}\, X_C|_{S_k^*}$ for every $k$.\\
At degree one we have
\[
X_C(\sigma_1)+\alpha_1=\alpha_1^*
\]
with $\alpha_1\in S_k^*$, thus one may choose $\sigma_1\in S_k^*$ such that $\alpha_1^*\in \ker X_{C_s}|_{S_k^*}$. At degree $j>1$ we have
\[
\begin{array}{rc}
&X_C(\sigma_j)+X_{h_2}(\sigma_{j-1})+\cdots+X_{h_j}(\sigma_1) +\alpha_1\sigma_{j-1}+\cdots +\alpha_{j-1}\sigma_1\\
& \quad =\alpha_1^*\sigma_{j-1}+\cdots +\alpha_{j-1}^*\sigma_1+\alpha_j^*,
\end{array}
\]
and all terms but the first on the left hand side and the last on the right hand side are automatically contained in $S_j^*$. As above, one may choose $\sigma_j\in S_j^*$ such that $\alpha_j^*\in \ker X_{C_s}|_{S_k^*}$.
\end{enumerate}

\end{proof}
\begin{remark}{\em
\begin{enumerate}[$(a)$]
\item The condition $g(0)=0$ does not exclude any interesting cases: When $g(0)\not=0$ then we may assume that $g$ is constant by the straightening theorem, and there remains $\partial f/\partial x_1=\lambda f$. With $\lambda \not=0$ and $f(0)=0$ this implies $f=0$.
\item To motivate the consideration of $\sigma f$ instead of $f$ in part $(c)$, note that $f$ and $\sigma f$ have the same local solution trajectories in the analytic setting; compare Proposition \ref{nonstatprop}(c).
\end{enumerate}
}
\end{remark}
The following example shows that one cannot sharpen the statement $(b)$ of Theorem \ref{linnormthm} in general.
\begin{example}\label{normalizerex}{\em   Let $A=A_s\not=0$ such that there exists a nonconstant homogeneous polynomial $\phi$ with $X_A(\phi)=0$, $\gamma$ a nonconstant series in one variable with $\gamma(0)=1$ and $f=\gamma(\phi)\cdot A$. Moreover let $h(x)=Cx$ be linear such that $C$ and $A$ commute but $X_C(\phi)\not=0$. Then
\[
[h,f]=X_C(\gamma)\gamma^{-1}\cdot f,
\]
hence $h$ lies in the normalizer of $f$ and satisfies the conclusion of part $(b)$, but $h$ does not lie in the centralizer of $f$.\\
For a particular example in dimension two set $A={\rm diag}\,(1,\,-1)$, $\phi=x_1x_2$ and $C=I_2$.
}
\end{example}
%%%%%%%%%%%%%%%%%%%%%%%%%%%%%%%%%%%%%%%%%%%%%%%%%%%%%%%%%%%%%%%%%%
\subsection{{PDNF in coordinate version}}
So far we worked with the coordinate-independent characterization of Poincar\'e-Dulac normal forms. In order to discuss the detailed structure of ${\mathcal C}_0^{\rm for}(A_s)$ and to distinguish various cases, we now employ eigencoordinates. Thus, from now on we assume that
\begin{equation}\label{diagmat}
A_s={\rm diag}\,(\lambda_1,\ldots,\lambda_n)\text{   with  } \lambda_1,\ldots,\lambda_n\in\mathbb C,
\end{equation}
with no loss of generality. We recall the coordinate-dependent characterization of PDNF:
\begin{lemma}\label{resonancelem}
Let $A_s$ be given as in \eqref{diagmat}. Then a power series lies in ${\mathcal C}_0^{\rm for}(A_s)$ if and only if it can be written as a series in vector monomials $x_1^{m_1}\cdots x_n^{m_n}e_j$ which satisfy the {\em resonance condition}
\begin{equation}\label{resonance}
\left<m,\lambda\right>-\lambda_j=0.
\end{equation}
\end{lemma}
Here we employ the familiar abbreviations $\left<m,\lambda\right>=\sum m_i\lambda_i$, and we will also later on use $|m|=\sum m_i$, for any row $(m_1,\ldots,m_n)$ with nonnegative integer entries.\\

For the remainder of the paper we make the following\\

\noindent{\bf Blanket assumptions.}
\begin{itemize}
\item We always let $A_s$ be as in \eqref{diagmat}.
\item For $f=A_s+\cdots$ in PDNF we stipulate that $A_n$ is strictly upper triangular.
\end{itemize}
\begin{remark}{\em
The second blanket condition poses no restriction (since it may always be achieved by a linear transformation) but it ensures that all the $\widetilde f=A_n+\sum f_j$ such that $f=A_s+\widetilde f$ is in  PDNF form a vector space. (This would not be the case otherwise). Given a positive integer $m$, the finite dimensional vector space of all such polynomial vector fields $\widetilde f=A_n+\sum_{i=2}^m f_j$ of degree $\leq m$ is sometimes identified with the space of its coefficients.
}
\end{remark}

For further analysis we introduce a distinction regarding the dimension of the centralizer of $A_s$. We recall from \cite{Wa91}, specifically Corollary 1.7:
\begin{lemma}\label{fireslem}
Let $A_s$ be as in \eqref{diagmat}. The dimension of  $\mathcal {C}^{\rm for}(A_s)$ $($as well as that of $\mathcal {C}(A_s)$$)$ is infinite if and only if:
\begin{equation}\label{fires}
\text{There are integers  } d_1,\ldots,\,d_n\geq 0 \text{  such that  }\left<d,\,\lambda\right>=0 \text{  and  } |d|>0.
\end{equation}
Equivalently, there exist nonnegative integers $d_i$ such that $|d|>0$ and the monomial $\phi:=x_1^{d_1}\cdots x_n^{d_n}$ is a first integral of $\dot x= A_sx$. In turn, this property is equivalent to the existence of some nonconstant formal first integral of $\dot x=A_sx$. {Finally this property is equivalent to the existence of some nonconstant formal first integral of $\dot x=Ax$.}
\end{lemma}
Below we will prove some ``generic'' results for vector fields in Poincar\'e-Dulac normal form, given $A_s$ such that nontrivial vector fields in PDNF exist. We need to define specific notions of genericity.
\begin{definition}\label{nfspecs}
{\em
 We say that a condition for PDNF with given $A_s$ as above {\em holds Z-generically} if there exists a positive integer $m$ such that the condition holds for a Zariski open and dense subset of all polynomial vector fields  $\widetilde f=A_n+\sum_{i=2}^m f_j$  with $f=A_s+\widetilde f$ in PDNF.
}
\end{definition}
For further investigations it seems appropriate to classify normal forms according to properties of the eigenvalues of $A_s$. The following is known from \cite{Wa91}:
\begin{proposition}\label{nofodecomposeprop}
There exists a direct sum decomposition $\mathbb C^n=U\oplus W$ with $A_s$-invariant (possibly trivial) subspaces $U$ and $W$.  The subspace $W$ is by definition the largest $A_s$-invariant subspace on which every polynomial first integral of $\dot x=A_sx$ is constant, and one obtains for $x=u+w\in U \oplus W$ a decomposition
\begin{equation}\label{nofodecompose}
\begin{array}{rcl}
\dot u&=&g(u)\\
\dot w&=&h(u,w).
\end{array}
\end{equation}
\end{proposition}
Moreover, by Lemma \ref{fireslem} the dimension of $\mathcal {C}^{\rm for}(A_s)$ is finite if and only if $\mathbb C^n=W$. This is the first case we will discuss. Later on we will consider the case $W=\{0\}$.
%%%%%%%%%%%%%%%%%%%%%%%%%%%%%%%%%%%%%%%%%%%%%%%%%%%%%%%%%%%%
\section{Finite dimensional centralizer of $A_s$}\label{findimsec}
We will obtain a rather clear and complete picture in the case when ${\mathcal C}^{\rm for}(A_s)$ has finite dimension, thus ${\mathcal C}^{\rm for}(f)$ necessarily has finite dimension by Proposition \ref{lincentprop}. We mostly consider the formal setting but will discuss convergence in some instances.

We first note an immediate consequence of Theorem \ref{linnormthm} and Lemma \ref{fireslem}.
\begin{proposition}\label{findimcentnorm}
Whenever $\dim \,{\mathcal C}^{\rm for}(A_s)<\infty$  then
\[
{\mathcal N}^{\rm for}(f)={\mathcal C}^{\rm for}(f)+\left\{\beta f;\,\beta \in \mathbb C[[x_1,\ldots,x_n]]\right\}
\]
for any $f=A+\cdots$ in PDNF.
\end{proposition}

There remains to investigate the centralizer.
 For the sake of a transparent exposition, we take a step-by-step approach to the characterization of centralizers. We dispose in advance of the simplest case, which is characterized by the property that every PDNF is necessarily trivial.
\begin{proposition}\label{simplecase} %{
For $f$ in PDNF, with $A_s$ as in \eqref{diagmat}, assume that $\lambda_1,\ldots,\lambda_n$ are pairwise different and there exists no resonance condition \eqref{resonance} with $|m|>1$. Then:
\begin{enumerate}[$(a)$]
\item One has $f=A_s$ and ${\mathcal C}_0^{\rm for}(f)$ {is an abelian Lie algebra of} dimension $n$.
\item The centralizer elements are precisely the linear vector fields given by diagonal matrices.
\end{enumerate}
%}
\end{proposition}
\begin{proof} By Lemma \ref{resonancelem} the PDNF is just $A=A_s$, and the only vector monomials commuting with $A_s$ are the $x_je_j$.
\end{proof}
In the setting of Proposition \ref{simplecase} we will also discuss convergence (which is not automatic even in this case). Before stating the result we remind the reader of a condition for the existence of a convergent normal form transformation.\\

\noindent {\em{\bf Condition Omega} (see Bruno \cite{Br89}):  For every positive integer $k$ let
\[
\omega_k:=\min\left\{|\left<m,\lambda\right>-\lambda_j|;\,|\left<m,\lambda\right>-\lambda_j|\not=0,\, 1\leq j\leq n,\,|m|\leq 2^k\right\}.
\]
We say that {\em Condition Omega} is satisfied if and only if $\sum2^{-k}\log\omega_k$ converges.
}
\begin{remark} {\em
\begin{enumerate}[(a)]
\item In the setting of Proposition \ref{simplecase} Condition Omega is sufficient for the existence of a convergent normalizing transformation.
\item Condition Omega holds in particular when the eigenvalues of $A$ are in a {\em Poincar\'e domain}, thus all $\lambda_i$ are one side of a suitable line through $0$ in the complex plane. In that case the centralizer of $A_s$ is necessarily finite dimensional, and there exists a convergent normalizing transformation even if there are relations $\left<m,\,\lambda\right>=0$ with $|m|>1$.
\end{enumerate}
}
\end{remark}

\begin{proposition}\label{simpleomega}Let $f=A+\cdots$ be a local analytic vector field.
\begin{enumerate}[$(a)$]
\item If  $\lambda_1,\ldots,\lambda_n$  are pairwise different and admit no resonance condition \eqref{resonance} with $|m|>1$, then $\dim\,{\mathcal C}_0(f)\leq n$.
\item{ If in addition the eigenvalues of $A$ satisfy Condition Omega then
\[
{\mathcal N}_0(f) ={\mathcal C}_0(f)+\left\{\beta\cdot f;\,\beta \text{ analytic in }0\right\}.
\]
}
\item If ${\mathcal C}_0(f)$ contains an element of the form
\[
g(x)={\rm diag}\,\left(\mu_1,\ldots,\mu_n\right)\,x+\cdots
\]
such that the $\mu_i$'s are pairwise different, admit no resonance \eqref{resonance} with $|m|\geq 2$, and satisfy Bruno's  Condition $\omega$, then $\dim\,{\mathcal C}_0(f)= n$.
\end{enumerate}
\end{proposition}
\begin{proof} Part $(a)$ is a direct consequence of Proposition \ref{simplecase}. We prove  part $(b)$.{ With Condition Omega there exists a convergent transformation of $f$ to normal form, hence we may assume that $f=A=A_s$. Given $g$ such that $[g,\,f]=\lambda f$ for some analytic $\lambda$. Then
\[
[g-\beta A,\,A]=\left(\lambda-X_A(\beta)\right)\,A
\]
and the assertion follows if there exists a convergent power series $\beta$ such that $X_A(\beta)=\lambda$. Writing $\lambda=\sum_{i\geq 1}\lambda_i$, the ansatz $\beta=\sum\beta_i$ yields the necessary and sufficient conditions
\[
X_A(\beta_j)=\lambda_j,\quad j\geq 1.
\]
Now the eigenvalues of $X_A|_{S_j}$ are just the $\left<m,\,\lambda\right>$ with $|m|=j$, and Condition Omega bounds the growth of the $\left<m,\,\lambda\right>^{-1}$, hence the growth of the operator norm of  $(X_A|_{S_j})^{-1}$, with $j$. Therefore $\beta$ has a nonempty domain of convergence if $ \lambda$ has.\\
Part (c) follows from \cite{Wa00b}, Addendum to Theorem 1.}
\end{proof}
{We note an implication for the global centralizer of an analytic vector field.
\begin{corollary}\label{simplecasecor} Let $f$ be analytic on an open and connected $U\subseteq\mathbb C^n$, and assume that  $f$ admits a stationary point $y$ such that the eigenvalues of $Df(y)$ satisfy the hypothesis of Proposition \ref{simplecase}.
 Then ${\mathcal C}_U(f)$ is abelian of dimension $\leq n$.
\end{corollary}
}
As a first  step beyond the setting of Proposition \ref{simplecase} we characterize the linear elements in the centralizer of $A$. Ordering possible multiple eigenvalues consecutively, we may assume that
\begin{equation}\label{linsemi}
A_s=\left(\begin{array}{cccc}
\mu_1I_{n_1} &  0             &  0        &   0   \\
 0           & \mu_2I_{n_2}   &  0       &   0   \\
 0           &   0            &  \ddots  &   0  \\
 0           &   0            &  0       &   \mu_\kappa I_{n_\kappa}  \\
\end{array}\right) ,
\end{equation}
where (generally) $I_{r}$ denotes the unit matrix of order $r$, and the $\mu_i$'s are pairwise different. By the blanket assumption we have
\begin{equation}\label{linstuff}
A=\left(\begin{array}{cccc}
\mu_1I_{n_1}+N_1 &  0             &  0        &   0   \\
 0           & \mu_2I_{n_2}+N_2   &  0       &   0   \\
 0           &   0            &  \ddots  &   0  \\
 0           &   0            &  0       &   \mu_\kappa I_{n_\kappa} +N_\kappa \\
\end{array}\right)
\end{equation}
with strict upper triangular matrices $N_i$. Any matrix commuting with $A$ is of block diagonal form with blocks $C_i$ of size $n_i\times n_i$ respectively, and $C_i$ commuting with $N_i$. From Lemma \ref{lincommutelem} in the Appendix we obtain:
\begin{lemma}\label{lincentlem}
Let $d$ be the dimension of the space of all linear vector fields commuting with \eqref{linstuff}. Then
\[
n\leq d\leq n_1^2+\cdots+n_\kappa^2.
\]
If all the $N_i$ have maximal rank then $d=n$; if all $N_i=0$ then $d=n_1^2+\cdots+n_\kappa^2$.
\end{lemma}
{
We now determine  bounds for the dimension of the centralizer ${\mathcal C}^{\rm for}(f)$ with $f$ in PDNF.
Set
\[
\mathcal R:=\bigcup\limits_{j=1}\limits^n\mathcal R_{j}, \quad \mathcal R_{j}:=\{m\in\mathbb Z_+;\ \ \langle m, \lambda\rangle=\lambda_j, \ |m|\ge 2\},
\]
and let $\mathfrak r$ be the number of elements in $\mathcal R$.
{
\begin{theorem}\label{finitedimT1}
Let $A_s$ be as in \eqref{linsemi}, $d$ as in Lemma \ref{lincentlem}, and $\dim\,{\mathcal C}^{\rm for}(A_s)<\infty$, hence ${\mathfrak r}<\infty$. Then for any vector field
\[
f(x)=Ax+p(x)
\]
in PDNF one has  $d\le\dim\, {\mathcal C}^{\rm for}(f)\le d+\mathfrak r$. In particular $\dim\, {\mathcal C}^{\rm for}(f)\geq n$.
\end{theorem}
\begin{proof} Proposition \ref{simplecase} takes care of the case $\mathfrak r=0$, therefore we may assume that $\mathfrak r>0$. {Moreover we may assume that $A$ is in Jordan canonical form.} Since the vector field is in PDNF, by the definition of $\mathcal R_j$ we have
\begin{equation}\label{paux}
p(x)=\sum\limits_{j=1}\limits^n\left(\sum\limits_{m_j\in\mathcal R_{j}} p_{m_j} x^{m_j}\right) e_j
\end{equation}
with $m_j=(m_{j1},\ldots,\,m_{jn})$, $p_{m_j}\in \mathbb C$ (possibly zero) and $e_j$ denoting the $j$th unit vector.
Let $g\in{\mathcal C}^{\rm for}(f)$, and write $g(x)=Bx+q(x)$, with $B$ linear and $q$ collecting the higher order terms.  By Proposition \ref{lincentprop} we have $[B,A_s]=0$, and
\begin{equation}\label{qaux}
q(x)=\sum\limits_{j=1}\limits^n\left(\sum\limits_{\ell_j\in\mathcal R_{j}} q_{\ell_j} x^{\ell_j}\right) e_j
\end{equation}
is a linear combination of resonant monomials.
Now $[g,\,f]=0$ is equivalent to
\[
[B,A_n]+[B,p]+[q,A_n]+[q,p]=0,
\]
and by separating linear from higher order terms this is equivalent to
\[
[B,A_n]=0 \text{  and  }[B,p]+[q,A_n]+[q,p]=0.
\]
In particular we have $[B,A]=0$, and
with linearly independent $C_1,\ldots,C_d$ that span the space of linear vector fields commuting with $A$ we get ${ B}=\sum { b_i}C_i$ with suitable scalars $b_i$. There remains
\[
[B,p]+[q,A_n]+[q,p]=0.
\]
{
Proceeding, we
let $\rho_i\subseteq\{1,\ldots,n\}$ denote the set of indices for which the corresponding rows of $C_i$ do not vanish, $i=1,\ldots,d$.  Then we have
\[
[B,p]=\sum_{i=1}^n b_i [C_i,p]=\sum_{j=1}^n\left(\sum\limits_{s\in (\rho_i \ni j)}\sum\limits_{m_s\in \mathcal R_{s}} p_{m_s}\sigma_{j,m_s}(b_1,\ldots,b_n) x^{m_s} \right) e_j
\]
by evaluation of the brackets, with linear homogeneous forms $\sigma_{j,m_s}$ (possibly zero) whose coefficients depend only on the entries of the $C_i$ and the exponents $\ell_j$ of $x^{\ell_j}$ appearing in $p(x)$. (Here the notation ${s\in (\rho_i \ni j)}$ means that $s$ runs over all $\rho_i$ that contain $j$.)}

Moreover we have
\[
[q,p]= \sum_{j=1}^n \sum\limits_{m_j\in\mathcal R_{j}}
\left(p_{m_j}q_{\ell_i}-q_{m_j}p_{\ell_i}\right)m_{ji}x^{m_j+\ell_i-e_i}e_j.
\]
{
Finally, by an argument similar to the one used for $[B,p]$  we find
\[
[q,A_n]=\sum_{j=1}^n\left(\sum_{m_k\in{\mathcal R}_j}a_{j,k}\tau_{j,\,m_k}\bigl((q_{m_\ell})_{m_\ell\in\mathcal R}\bigr)x^{m_k}\right)e_j
\]
with $a_{j,k}\in\{0,1\}$ and linear homogeneous forms $\tau_{j,m_k}$ whose coefficients depend only on the exponents $\ell_j$ of $x^{\ell_j}$ that may appear in $q$, which are known. (Recall that the nilpotent part $A_n$ of $A$ is in Jordan canonical form.)}

To summarize, we have a system of homogeneous linear equations for the $b_i$'s and the $q_{m_j}$'s.
Since the number $r$ of the resonant monomials appearing with nonzero coefficient in $[B,p]+[q,A_n]+[q,p]$ is at most equal to $\mathfrak r$, it follows that $[g,f]=0$ if and only if the coefficients $b_{i}$'s and $q_{m_j}$'s of $g$ satisfy $r$ linear homogeneous equations. The coefficient matrix of this linear homogeneous equation system has size $r\times (d+\mathfrak r)$ and if  its rank equals $r^*\le r$, then the system has exactly $d+\mathfrak r-r^*\ge d+\mathfrak r-r\ge d$ linearly independent solutions. These linearly independent solutions correspond to a basis of ${\mathcal C}^{\rm for}(f)$. Consequently  $d+\mathfrak r\geq \mbox{\rm dim}\,{\mathcal C}^{\rm for}(f)=d+\mathfrak r-r^*\ge d$.
\end{proof}
}
}
Note that the upper bound for the dimension can be attained only in the case that all $p_{m_j}=0$, i.e. $p(x)\equiv 0$.
The result for the PDNF implies properties of general vector fields.
\begin{corollary}
Let $h(x) = Bx +\sum_{j\geq 2} h_j(x)\in \mathbb C[[x_1,\ldots,x_n]]^n$ be a formal vector field such that ${\mathcal C}^{\rm for}(B_s)$ is finite dimensional, and $B$ conjugate to $A$ as in \eqref{linstuff}. Then
\[
{\rm dim}\,{\mathcal C}^{\rm for}(h)\geq d\geq n,
\]
with $d$ as in Lemma \ref{lincentlem}.
\end{corollary}
\begin{remark}{\em
Since there exist nontrivial lower bounds for $\dim\, {\mathcal C}^{\rm for}(f)$ which depend only on the space dimension $n$, it is natural to ask about upper bounds. The following example shows that there cannot exist an upper bound depending only on $n$ whenever $n>2$.
}
\end{remark}

\begin{example}\label{eg2}{\em
\begin{enumerate}[$(a)$]
\item We first consider an example in dimension three. Let $q$ be a positive integer and set
\[
A=A_s={\rm diag}\,\left(12q,\,3,\,2\right).
\]
Evaluating the resonance conditions \eqref{resonance} one has $12 q m_1+3 m_2+2m_3=2$ for the third entry, with the only  solution $(0,0,1)$, and similarly $12 q m_1+3 m_2+2m_3=3$ for the second entry, with the only  solution $(0,1,0)$. Nontrivial solutions appear only in the resonance condition for the first entry, which is $12q m_1+3m_2+2m_3=12q$, and they are of the form
\[
\left(0,\,4q-2k,\,3k\right),\quad 0\leq k\leq 2q.
\]
Therefore every PDNF has the form
\begin{equation}\label{eg2e1}
f=A_s+\sum_{k=0}^{2q}\alpha_kp_k,\quad \text{with  }p_k(x)=\begin{pmatrix}x_2^{4q-2k}x_3^{3k}\\0\\0\end{pmatrix}.
\end{equation}
The $p_i$'s commute with $A_s$ by construction, and they commute pairwise since the first column of the Jacobian of any $p_i$ is zero, hence $Dp_i(x)p_j(x)=0$ for all $i$ and $j$. Obviously $A_s$ and the $p_i$ form a linearly independent system, and they are all contained in the formal centralizer of the vector field \eqref{eg2e1}. This shows $\dim\,{\mathcal C}^{\rm for}(f)\geq 2q+2$.\\
Note that $f$ admits  the meromorphic first integral $\rho:=x_2^2/x_3^3$.
\item In dimension $n>3$ let $\mu_1,\ldots,\mu_{n-3}$ be such that $1, \,\mu_1,\ldots,\mu_{n-3}$ are linearly independent over the rationals $\mathbb Q$, and set
\[
A=A_s={\rm diag}\,\left(12q,\,3,\,2,\,\mu_1,\,\ldots,\mu_{n-3}\right).
\]
Then, slightly modifying the arguments above, one obtains the same estimate $\dim\,{\mathcal C}^{\rm for}(f)\geq 2q+2$, since no further nontrivial resonance conditions appear.
\end{enumerate}
}
\end{example}
The next example illustrates that every centralizer dimension between the lower and upper estimate in Theorem \ref{finitedimT1} may be attained.
{
\begin{example}\label{eg3}{\em
We specialize Example \ref{eg2}~$(a)$ by setting
\[
A_s={\rm diag}\,\left(12,\,6,\,3\right).
\]
Then every PDNF has the form
\begin{equation}\label{eg3e1}
f=A_sx+p(x), \quad p(x)=\alpha_1 p_1+\alpha_2p_2+\alpha_3p_3+\alpha_4p_4,
\end{equation}
with
\[
p_1(x)= x_2^2 e_1,\quad
p_2(x)= x_2 x_3^2 e_1,\quad
p_3(x)= x_3^4 e_1,\quad
p_4(x)= x_3^2 e_2.
\]
This system has always the elementary first integral $H=x_2x_3^{-2}-\dfrac{\alpha_4}{3}\log x_3$.

For $g=B+q(x)\in{\mathcal C}^{\rm for}(f)$, then $B=\mbox{\rm diag}(b_{11},b_{22},b_{33})$ and
\[
q(x)=\beta_1x_2^2 e_1 +\beta_2 x_2 x_3^2 e_1 + \beta_3 x_3^4 e_1 +\beta_4  x_3^2 e_2.
\]
We now need to determine the seven coefficients $c=(b_{11}, b_{22}, b_{33}, \beta_1, \beta_2, \beta_3, \beta_4)$ such that
$[g,f]=0$, which is equivalent to $[B,p]+[q,p]=0$. Since
\begin{align*}
[B,p]&=\left(\begin{array}{c}
(2b_{22}-b_{11}) \alpha_1 x_2^2+(b_{22}+2b_{33}-b_{11})\alpha_2 x_2x_3^2+(4b_{33}-b_{11})\alpha_3 x_3^4\\
(2b_{33}-b_{22})\alpha_4 x_3^2\\
0
\end{array}\right),\\
[q,p]&=\left(\begin{array}{c}
2( \alpha_1\beta_4-\alpha_4\beta_1)x_2x_3^2+(\alpha_2\beta_4-\alpha_4\beta_2)x_3^4\\
0\\
0
\end{array}\right),
\end{align*}
one gets
\begin{equation}\label{eg3e}
\left(\begin{array}{ccccccc}
-\alpha_1  &  2 \alpha_1   &  0  &   0  &  0  &  0  &  0 \\
-\alpha_2 &   \alpha_2   &  2\alpha_2   & -2\alpha_4  &  0   &  0   &  2 \alpha_1 \\
-\alpha_3  &  0           &  4\alpha_3  &   0         & -\alpha_4  &  0  & \alpha_2\\
0         &  -\alpha_4  &  2 \alpha_4   &  0   &  0  &  0  &  0
\end{array}\right)
c^\tau
=0.
\end{equation}
The coefficient matrix of equation \eqref{eg3e} has rank $4$ generically, and also rank $3$ or $2$ or $1$ or $0$ for suitable choice of the entries  in this matrix. In fact, by solving equations \eqref{eg3e} in the $b_{ii}$'s and $\beta_j$'s we can reach the next conclusions:
\begin{itemize}
\item If $\alpha_1,\ \alpha_4\ne 0$, then $b_{11}=2b_{22}, \ b_{33}=\dfrac 12 b_{22}$, $\beta_1=\dfrac{\alpha_1}{\alpha_4}\beta_4,\ \beta_2=\dfrac{\alpha_2}{\alpha_4}\beta_4$, and consequently $\mbox{\rm dim}\,{\mathcal C}^{\rm for}(f)=3$.
\item If $\alpha_1=0, \ \alpha_4\ne 0$, then $b_{22}=2b_{33}$, $\beta_1=\dfrac{\alpha_2}{2\alpha_4}(4b_{33}-b_{11}),\ \beta_2=\dfrac{\alpha_3}{\alpha_4}(4b_{33}-b_{11})+\dfrac{\alpha_2}{\alpha_4}\beta_4$, and consequently $\mbox{\rm dim}\,{\mathcal C}^{\rm for}(f)=4$.
\item If $\alpha_1\ne 0, \ \alpha_4= 0$, then $b_{11}=2b_{22}$. Furthermore
\begin{itemize}
 \item if $2\alpha_1\alpha_3-\alpha_2^2\ne 0$, then $b_{33}=\dfrac 12 b_{22}$ and $\beta_4=0$, and consequently $\mbox{\rm dim}\,{\mathcal C}^{\rm for}(f)=4$;
\item  if $2\alpha_1\alpha_3-\alpha_2^2= 0$, then  $\beta_4=-\dfrac{\alpha_2}{2\alpha_1}(2b_{33}-b_{22})$, and consequently $\mbox{\rm dim}\,{\mathcal C}^{\rm for}(f)=5$.
 \end{itemize}
\item For $\alpha_1= 0, \ \alpha_4= 0$,
\begin{itemize}
 \item if $\alpha_2 \ne 0$, then $b_{11}=  b_{22}+2b_{33}$ and $\beta_4=-\dfrac{\alpha_3}{\alpha_2}(2b_{33}-b_{22})$, and consequently ${\rm dim}\,{\mathcal C}^{\rm for}(f)=5$;
 \item if $\alpha_2 = 0$ and $\alpha_3\ne 0$ then $b_{11}= 4b_{33}$, and consequently ${\rm dim}\,{\mathcal C}^{\rm for}(f)=6$;
\item if $\alpha_2 = 0$ and $\alpha_3= 0$, the dimension is the maximum one, i.e. ${\rm dim}\,{\mathcal C}^{\rm for}(f)=7$.
      \end{itemize}
\end{itemize}
}
\end{example}

Next we give an improved estimate for the case when $A=A_s$ as given by \eqref{linsemi} has multiple eigenvalues. Before stating the result, we introduce the notation
\[
{\mathcal R}^*_k:={\mathcal R}_j \text{ when  }\mu_k=\lambda_j, \, {\mathfrak r^*}_k:=\#{\mathcal R}^*_k,\quad1\leq k\leq \kappa.
\]
Recall that all ${\mathfrak r^*_k}$ are finite since ${\mathcal C}^{\rm for}(A_s)$ is finite dimensional.

\begin{theorem}\label{finitedimT2} Let $A=A_s$ be given by \eqref{linsemi}, and $f=A_s+\cdots$ in PDNF.
 Then
\[
n_1^2+\ldots+n_{\kappa}^2\le \mbox{\rm dim}\,{\mathcal C}^{\rm for}(f)\le n_1(n_1+\mathfrak r_1^*)+\ldots+n_\kappa(n_\kappa+ \mathfrak r_\kappa^*).
\]
\end{theorem}
\begin{proof}
The arguments are similar to those in the proof of Theorem \ref{finitedimT1}; here we indicate only the differing parts. Under our assumption the total number of nonlinear resonant monomials equals $\mathfrak E_r:= n_1 \mathfrak r_1^* +\ldots+n_\kappa \mathfrak r_\kappa^*$. We also define $\mathfrak L_c:=n_1^2+\ldots+n_\kappa^2$.

For $g(x)=Bx+q(x)\in {\mathcal C}^{\rm for}(f)$ with $q(x)$ a linear combination of nonlinear resonant monomials, Lemma \ref{lincommutelem} shows that
\[
B=\mbox{\rm diag}(B_1,\ldots,B_\kappa)
\]
with $B_j$ being an arbitrary $n_j\times n_j$ matrix, $j=1,\ldots,\kappa$. Furthermore, $[g,f]=0$ if and only if $[B,p]+[q,p]=0$. The latter is equivalent to a system of at most $\mathfrak E_r$ linear homogeneous equations for $\mathfrak E_r+\mathfrak L_c$ unknowns, which are the entries of the $B_j$'s and the coefficients of the monomials appearing in $q(x)$. By construction, the coefficients of this system of linear homogeneous equations are linear homogeneous functions in the coefficients of $p(x)$. So if $p(x)\equiv 0$ then $\mbox{\rm dim}\,{\mathcal C}^{\rm for}(f)$ attains the maximal possible value $\mathfrak E_r+\mathfrak L_c$. The lower estimate was already shown in Theorem \ref{finitedimT1}. This completes the proof of the theorem.
\end{proof}

We present an application of Theorem \ref{finitedimT2}.
\begin{example} \label{eg4}{\em
{\em Let $A_s=\mbox{\rm diag}(12,12,6,6,6,3)$. The nonlinear term of a vector field $f=A_sx+p(x)$ in PDNF is of the form
\[
p(x)=\left(
p_1(x), \
p_2(x), \
p_{31} x_6^2, \
p_{41} x_6^2, \
p_{51} x_6^2,\
0
\right)^\tau
\]
with $p_j(x)=p_{j1}x_3^2+p_{j2}x_4^2+p_{j3}x_5^2+p_{j4}x_3x_4+p_{j5}x_3x_5+p_{j6}x_4x_5+p_{j7}x_3x_6^2+p_{j8}x_4x_6^2+p_{j9}x_5x_6^2+p_{j10} x_6^4$, $j=1,2$, thus we have $23$ resonant monomials. For $g(x)=Bx+q(x)\in {\mathcal C}^{\rm for}(f)$, with
\[
B=\mbox{\rm diag}\,(B_1,B_2,b_{66}),\quad
B_1=\left(\begin{array}{cc}
b_{11} & b_{12}\\
b_{21} & b_{22}
\end{array}\right), \ \
B_2=\left(\begin{array}{ccc}
b_{33} & b_{34}  & b_{35}\\
b_{43} & b_{44}  & b_{45}\\
b_{53} & b_{54}  & b_{55}
\end{array}\right),
\]
and $q(x)=\left(q_1(x),\
q_2(x),\
q_{31} x_6^2,\
q_{41} x_6^2,  \
q_{51} x_6^2,\
0
\right)^{\tau}$ with $q_i$ being a linear combination of the same monomials as $p$ with coefficients $q_{jk}$ instead of $p_{jk}$, $j=1,\,2$.

Now $0=[g,f]=[B+q,p]$ is equivalent to the vanishing of all coefficients of the $23$ resonant monomials in $[B+q,p]$. This yields a  system of $23$ linear homogeneous equations $($whose expressions are quite bulky and are omitted here$)$ for $37$ unknowns. By Theorem \ref{finitedimT2}, in this case $14\le \mbox{\rm dim}\,{\mathcal C}^{\rm for}(f)\le 37$.\qed
}
}
\end{example}
}
{Combining the arguments in the proofs of the last two theorems, we arrive at the following improvement of Theorem \ref{finitedimT1}. We do not carry out the details of the proof.

\begin{corollary}\label{finitedimcor} Let $A$ be given by \eqref{linstuff}, $A$ not diagonal, such that all $N_i$ have maximal rank. Then
\[
 d \le \mbox{\rm dim}\,{\mathcal C}^{\rm for}(f)\le n+ n_1 \mathfrak r_1 +\ldots+n_\kappa \mathfrak r_\kappa.
\]
\end{corollary}

\begin{remark}{\em
Note that $[q,A]=[q,A_n]$  is in general not zero for nonlinear $q$ in the centralizer of $A_s$. Therefore,  even when $p(x)\equiv 0$, the space of the elements  in the centralizer with order greater than one has dimension less than $n_1 \mathfrak r_1 +\ldots+n_\kappa \mathfrak r_\kappa$, since $ [q,A_n]=0$  appears as an additional condition. In the next example, we illustrate that the lower bound in Theorem \ref{finitedimT1} can be reached, $[q,A_n]\ne 0$, and the optimal upper bound is less than that given in Corollary \ref{finitedimcor}.
}\end{remark}

\begin{example}\label{eg4} {\em Consider $f=A+p$ in PDNF with
\[
A=\left(\begin{array}{cccccc}
3 &  1  & 0  &  0  &   0   &  0\\
0 &  3  & 1  &  0  &   0   &  0\\
0 &  0  & 3  &  0  &   0   &  0\\
0 &  0  & 0  &  2  &   1   &  0\\
0 &  0  & 0  &  0  &   2   &  0\\
0 &  0  & 0  &  0  &   0   &  1
\end{array}\right),
\]
then $p(x)$ is in general of the form
\[
p(x)=\left(\begin{array}{c}
a_1x_4x_6+a_2 x_5x_6+a_3 x_6^3\\
a_4x_4x_6+a_5 x_5x_6+a_6 x_6^3\\
a_7x_4x_6+a_8 x_5x_6+a_9 x_6^3\\
a_{10} x_6^2\\
a_{11} x_6^2\\
0
\end{array}\right).
\]
For $g=B+q\in {\mathcal C}^{\rm for}(f)$, then one has
\[
B=\left(\begin{array}{cccccc}
\alpha_1 &  \alpha_2  & \alpha_3  &  0  &   0   &  0\\
0 &  \alpha_1  & \alpha_2  &  0  &   0   &  0\\
0 &  0  & \alpha_1  &  0  &   0   &  0\\
0 &  0  & 0  &  \alpha_4  &   \alpha_5   &  0\\
0 &  0  & 0  &  0  &   \alpha_4   &  0\\
0 &  0  & 0  &  0  &   0   &  \alpha_6
\end{array}\right),
\]
and $q(x)$ in general has the same form as $p(x)$ replacing $a_j$'s by $b_j$'s.

Direct calculations show that
\[
[q,A_n]=\left(\begin{array}{c}
b_4 x_4 x_6 -( b_1 - b_5) x_5 x_6 + b_6 x_6^3\\
 b_7 x_4 x_6 - (b_4 - b_8) x_5 x_6 + b_9 x_6^3\\
  -b_7 x_5 x_6\\
   b_{11} x_6^2\\
   0 \\
    0
    \end{array}\right),
\]
which consists of resonant monomials, but does not vanish in general. Solving
\begin{equation}\label{eg4e}
0=[g,f]=[B,p]+[q,A]+[q,p]=[B,p]+[q,A_n]+[q,p]
\end{equation}
for generic $a_i$'s yields
\[
B=\left(\begin{array}{cccccc}
 3 c_4/2 & c_2  & c_3  &  0  & 0 & 0\\
 0     &  3 c_4/2   & c_2  &  0  & 0 & 0\\
 0 & 0 & 3 c_4/2  &  0 & 0  & 0\\
 0 &  0 & 0 & c_4 & c_2 & 0\\
  0 & 0 & 0 & 0  & c_4 & 0\\
  0 & 0 &   0 & 0 & 0 & c_4/2
  \end{array}\right),
  \]
\[
 q(x)=\left(\begin{array}{c}
  b_1 x_4 x_6 + b_2 x_5 x_6 +   b_3 x_6^3\\
   (a_4 c_2 + a_7 c_3) x_4 x_6 + (b_1 - a_1 c_2 + a_5 c_2 + a_8 c_3) x_5 x_6 \\
  + (a_6 c_2 + a_{10} (b_1 - a_1 c_2) + a_9 c_3 +   a_{11} (b_2 - a_2 c_2 - a_1 c_3)) x_6^3\\
 a_7 c_2 x_4 x_6 + (a_8 c_2 + a_7 c_3) x_5 x_6 \\
   + (a_9 c_2 + a_{10} a_7 c_3 +
     a_{11} (b_1 - a_1 c_2 - a_4 c_3 + a_8 c_3)) x_6^3\\
      (a_{10} c_2 + a_{11} c_3) x_6^2 \\
 a_{11} c_2 x_6^2\\
  0
 \end{array}\right).
\]
In this case $\mbox{\rm dim}\,{\mathcal C}^{\rm for}(f)=6$, the minimum estimate given in Corollary \ref{finitedimcor}.

Solving \eqref{eg4e} with $a_i=0$ for all $i$ yields
\[
q(x)=\left(\begin{array}{c}
 b_1 x_4 x_6 + b_2 x_5 x_6 + b_3 x_6^3\\ b_1 x_5 x_6\\ 0 \\ b_{10} x_6^2\\ 0\\ 0
\end{array}\right),
\]
and $B$ as it is in the general form. This illustrates that the maximum dimension of centralizers ${\mathcal C}^{\rm for}(f)$ of the PDNF systems $f$ with the given $A$ is $6+4=10$, which is much smaller than the upper estimate  $6+(3\times 3+2\times 1)=17$ from Corollary \ref{finitedimcor}.\qed
}\end{example}
}
{To finish this section we return to the local analytic case.
\begin{proposition}\label{PDc}
Let the eigenvalues of $A$ be in a Poincar\'e domain. Then ${\mathcal C}_0^{\rm for}(f)$ is finite dimensional, and agrees with the analytic centralizer ${\mathcal C}_0(f)$. Moreover
\[
{\mathcal N}_0(f)={\mathcal C}_0(f)+\left\{\beta\,f;\,\beta\text{ analytic in  }0\right\}.
\]
\end{proposition}
\begin{proof} Since the eigenvalues lie in a Poincar\'e domain, there exists a convergent transformation to normal form. The formal centralizer of the normal form consists of polynomials, hence only of analytic vector fields. Therefore the analytic and the formal centralizer of $f$ are equal. Now Theorem \ref{linnormthm}~$(a)$ and the proof of
Proposition \ref{simpleomega}~$(b)$ show the assertion about the normalizer.
\end{proof}
}
%%%%%%%%%%%%%%%%%%%%%%%%%%%%%%%%%%%%%%%%%%%%%%%%
%%%%%%%%%%%%%%%%%%%%%%%%%%%%%%%%%%%%%%%%%%%%
\section{Infinite dimensional centralizer of $A_s$}\label{infidimsec}
For the case of an infinite dimensional centralizer of $A_s$ it seems considerably harder to characterize the centralizer of $f=A+\cdots$ in PDNF. We obtain partial results in the general case and obtain a rather precise description in an algebraically distinguished setting (see subsection \ref{subsec42}). A further discussion will appear in future work.
\subsection{{General structure}}

By Lemma \ref{fireslem} there exists a nonconstant first integral $x_1^{d_1}\cdots x_n^{d_n}$ for $\dot x= A_sx$, and we will make use of this property in the following. On some occasions we replace \eqref{fires} by a stronger assumption, viz.
\begin{equation}\label{fullfires}
\text{There exist integers  }d_1>0,\ldots,d_n>0\text{  such that  } \sum d_i\lambda_i=0.
\end{equation}
Using the terminology from Proposition \ref{nofodecomposeprop}, this assumption is equivalent to $W=\{0\}$.
\\

 First we record an immediate consequence of Example \ref{normalizerex}.
\begin{remark}{\em
When the centralizer of $A_s$ has infinite dimension then there exist vector fields $f=A_s+\cdots$ in PDNF such that
\[
{\mathcal C}^{\rm for}(f)+\{\beta f; \,\beta \in\mathbb C[[x_1,\ldots,x_n]]\} \subsetneqq {\mathcal N}^{\rm for}(f).
\]
}
\end{remark}
In the following we will focus on the centralizer. Nontrivial first integrals of $A_s$ enable reduction by invariants, and we recall some known facts here for the reader's convenience. We first rephrase \cite{Wa91}, Proposition 1.6:
\begin{lemma}\label{torfinlem} Let $B\in \mathbb C^{n\times n}$ be semisimple. Then the following hold.
\begin{enumerate}[$(a)$]
\item The algebra
\[
I(B):=\left\{\phi\in\mathbb C[x_1,\ldots,x_n];\,X_B(\phi)=0\right\}
\]
is finitely generated.
\item For every $\alpha\in\mathbb C$ the $I(B)$-module
\[
I_\alpha(B):=\left\{\psi\in\mathbb C[x_1,\ldots,x_n];\,X_B(\psi)= \alpha\psi\right\}
\]
is finitely generated.
\item The $I(B)$-module
\[
\mathcal{C}^{\rm pol}(B):=\left\{g\in\mathbb C[x_1,\ldots,x_n]^n;\,[B,\,g]=0\right\}
\]
 is finitely generated.
\end{enumerate}
\end{lemma}
{
\begin{remark}\label{torfinrem}{\em We recall a few more facts about these structures, assuming that the semisimple linear map $B$ is in diagonal form with eigenvalues $\mu_1,\ldots,\mu_n$.
\begin{enumerate}[$(a)$]
\item $I(B)$ is spanned over $\mathbb C$ by all monomials $x^m$ with $ \langle m,\mu\rangle=0,\ \ m\in\mathbb Z_+^n,\,  |m|\ge 1$, where $\mu$ is the $n$--tuple of eigenvalues of $B$. There exists a system of algebra generators of $I(B)$ that consists of monomials, but in general there will exist no algebraically independent system (equivalently, functionally independent system, as can be seen from Shafarevich \cite{Sha}) of algebra generators for $I(B)$. To determine a minimal generator set, one may (and in general has to) resort to methods from algorithmic algebra; in particular Dickson's lemma and Groebner bases are useful here $($see Cox et al. \cite{CLOS}$)$.
\item As a vector space over $\mathbb C$, $I_\alpha(B)$ is spanned by all monomials $x^k$ with $k\in\mathbb Z_+^n$ and $\left<k,\,\mu\right>=\alpha$.
\item The module ${\mathcal C}^{\rm pol}(B)$ spanned as a $\mathbb C$-vector space by all $x^\ell e_j$ with $x^\ell\in I_{\lambda_j}(B)$, and in particular contains every $Q_j$ with $Q_j(x):=x_je_j$, $1\leq j\leq n$.
\end{enumerate}
}
\end{remark}
The structure of the module ${\mathcal C}^{\rm pol}(B)$ may be quite complicated, and even minimal sets of generators may be very large. But the structure is particularly simple given the conditions described next.
\begin{lemma}\label{trivresonlylem}
Let $B={\rm diag}(\mu_1,\ldots,\mu_n)$, with all $\mu_i\not=0$, and assume that
\begin{equation}\label{trivresonly}
\left<m,\mu\right>-\mu_j=0 \text{ for } m\in \mathbb Z_+^n \Rightarrow m_j>0,\quad 1\leq j\leq n.
\end{equation}
Then ${\mathcal C}^{\rm pol}(B)$ is a free $I(B)$-module of rank $n$, generated by the $Q_j(x)=x_je_j$, $1\leq j\leq n$.
\end{lemma}
\begin{proof}
Given a vector monomial $p(x)=x_1^{m_1}\cdots x_n^{m_n}e_j\in {\mathcal C}_{\rm pol}(B)$ we have $m_j>0$, hence
\[
p(x)=\phi(x)\,Q_j(x),\quad \phi(x)=x_1^{m_1}\cdots x_j^{m_{j}-1}\cdots x_n^{m_n}\in I(B),
\]
which shows that the $Q_j$ generate the module. Obviously the $Q_j$ form a free system (even over $\mathbb C[x_1,\ldots,x_n]$).
\end{proof}
}
\begin{remark}{\em
In general the rank of the $I(B)$-module  ${\mathcal C}^{\rm pol}(B)$ may be larger than the space dimension $n$. Section \ref{findimsec} contains many examples for the case $I(B)=\mathbb C$, and a simple example with nontrivial $I(B)$ is given by $B={\rm diag}\,(1,\,-1,\sqrt 2,\,2\sqrt 2)$.
}
\end{remark}

{ From \cite{Wa91}, Proposition 3.4 and Theorem 3.6, we recall symmetry reduction by invariants:
\begin{proposition}\label{nofofinprop}
Let $f=A+\cdots $ be in PDNF, and denote by $\phi_1,\ldots,\phi_r$ a generator system for $I(A_s)$.  Let $g\in\mathcal{C}^{\rm for}(A_s)$ and
\[
\Phi:=\begin{pmatrix} \phi_1\\ \vdots \\ \phi_{r}\end{pmatrix}.
\]
Then there exists a formal vector field $\widehat g$ in $r$ variables such that the identity
\[
D\Phi(x)g(x)=\widehat g\left(\Phi(x)\right)
\]
holds, hence $\Phi$ is $($formally$)$ solution-preserving from $\dot x=g(x)$ to $\dot y=\widehat g(y)$. Whenever $A_s\not=0$ then the dimension of the Zariski closure $\overline{\Phi(\mathbb C^n)}$ is smaller than $n$. {Moreover $\widehat g=0$ on $\overline{\Phi(\mathbb C^n)}$ if and only if every first integral of $A_s$ is also a first integral of $g$.}
\end{proposition}
}
\begin{corollary}\label{reducomm}
{Let $f$ and $\Phi$ be as in Proposition \ref{nofofinprop}. Given $g\in{\mathcal C}^{\rm for}(A_s)$ with $[g,\,f]=\alpha f$ and $X_{A_s}(\alpha)=0$, there exist formal vector fields $\widehat f$ and $\widehat g$ in $r$ variables and a formal series $\widehat\alpha$  in $r$ variables such that $\alpha=\widehat\alpha\circ\Phi$ and }
\[
D\Phi(x)f(x)=\widehat f\left(\Phi(x)\right), \quad D\Phi(x)g(x)=\widehat g\left(\Phi(x)\right)
\]
and
\[
[\widehat g ,\,\widehat f]=\widehat\alpha\widehat f\text{  on  the Zariski closure of  }\Phi(\mathbb C^n).
\]
\end{corollary}
\begin{proof}
The existence of $\widehat f$ and $\widehat g$ follows immediately from Proposition \ref{nofofinprop}, and the commutator property is a consequence of the general rule
\[
D\Phi(x)\left[f,\,g\right](x)=[\widehat f,\,\widehat g](\Phi(x));
\]
moreover use $\alpha=\widehat\alpha\circ\Phi$.
\end{proof}

Note that the additional condition on $\alpha$ in the statement of the Corollary can always be realized due to Theorem \ref{linnormthm}.
{
\begin{remark}\label{rigidrem}{\em
The reduction is particularly useful in applications when $\widehat f$ has the property
\begin{equation}\label{rigid}
 [\widehat f,\widehat g]=0\Rightarrow\widehat g\in\mathbb C\widehat f.
\end{equation}
  By {Proposition} \ref{nofofinprop} this reduces the discussion of ${\mathcal C}_0^{\rm for}(f)$ to commuting formal vector fields which admit $\phi_1,\ldots,\phi_r$ as first integrals.\\
 The series expansion of $\widehat f$ has vanishing semisimple linear part, and we conjecture that for (in some sense)  ``generic'' $f$  with given $A_s$, property \eqref{rigid} will always hold, but a general proof seems very hard. Below we will discuss a distinguished class of vector fields for which the property does hold.
}
\end{remark}
}

We record some consequences of conditions \eqref{fires} and \eqref{fullfires}.
\begin{lemma}\label{estprep}
 Let $A_s$ satisfy \eqref{fires}, and moreover set
\[
\dim_{\mathbb Q}\left(\mathbb Q\lambda_1+\cdots+\mathbb Q\lambda_n\right)=:q.
\]
%{
\begin{enumerate}[$(a)$]
\item There exists  a $\mathbb Q$-basis $\nu_1,\ldots,\nu_q$ of $\mathbb Q\lambda_1+\cdots+\mathbb Q\lambda_n$ such that
\[
A_s=\nu_1C_1+\cdots+\nu_qC_q
\]
with matrices
\[
C_j={\rm diag}\,\left(c_{j1},\ldots,c_{jn}\right)
\]
that have integer entries. Moreover the $C_j$'s are linearly independent over $\mathbb C$ as well as over the field $\mathbb F$ of formal meromorphic functions.
\item Given nonzero $m$ with nonnegative integer entries, a resonance condition $\left<m,\,\lambda\right>-\lambda_k=0$ holds if and only if
\[
\left<m,\,(c_{j1},\ldots,c_{jn})\right>-c_{jk}=0,\quad 1\leq j\leq q.
\]
Likewise, one has $\left<m,\,\lambda\right>=0$  if and only if
\[
\left<m,\,(c_{j1},\ldots,c_{jn})\right>=0,\quad 1\leq j\leq q.
\]
\item If condition \eqref{fullfires} holds then there exist
\[
\left(\ell_{i1},\ldots,\ell_{in}\right),\quad 1\leq i\leq n-q
\]
with positive integer entries that are linearly independent over $\mathbb Q$, with $\left<(\ell_{i1},\ldots,\ell_{in}),\,\lambda\right>=0$ for all $i$; equivalently
\[
\left<(\ell_{i1},\ldots,\ell_{in}),(c_{j1},\ldots,c_{jn})\right>=0\text{  for all  }i,\,j.
\]
\item Assuming that \eqref{fullfires} holds, let
\[
\psi_i:=x_1^{\ell_{i1}}\cdots x_n^{\ell_{in}},\quad 1\leq i\leq n-q.
\]
If $D$ is a diagonal matrix such that $X_D(\psi_i)=0$ for all $i$ then $D$ is a linear combination of $C_1,\ldots, C_q$ over $\mathbb C$.
\end{enumerate}
%}
\end{lemma}
\begin{proof}
%{
Part $(a)$ is most easily seen by assuming w.l.o.g.\ that $\lambda_1,\ldots,\lambda_q$ are linearly independent over $\mathbb Q$ and choosing $\nu_1\in \mathbb Q\lambda_1,\ldots,\nu_q\in \mathbb Q\lambda_q$. The matrix $C_j$ then has the form
\[
C_j={\rm diag}\,\left(\delta_{j1}a_j,\ldots,\delta_{jq}a_j,*,\ldots,*\right)
\]
with integers $a_j\not=0$ and the Kronecker symbol $\delta_{jk}$, and the asterisks $*$ denoting integers.
Part $(b)$ is clear. As for part $(c)$, complete the nonzero vector $(d_1,\ldots,d_n)$ to a basis of the solution space of $\sum z_i\lambda_i=0$ over $\mathbb Q$, w.l.o.g.\ with integer entries. If such an additional basis element has a non-positive entry then add a suitable integer multiple of $(d_1,\ldots,d_n)$. For the final assertion let $D={\rm diag}\,(\theta_1,\ldots,\theta_n)$. Then $X_D(\psi_i)=0$ if and only if $\sum \ell_{ij}\theta_j=0$. Now the assertion follows from rank considerations: The matrix $\left(\ell_{ij}\right)$ has rank $n-q$, hence its kernel has dimension $q$ and contains the linearly independent columns $\left(\theta_{j1},\ldots,\theta_{jn}\right)^\tau$, $1\leq j\leq q$.
%}
\end{proof}

\begin{example}{\em
Given $q\in\{1,\dots,n-1\}$, for different choices of $A_s$ the minimal numbers of generators of $I(A_s)$ and $\mathcal C^{for}(A_s)$ may vary strongly. We illustrate the cases with $n=4$ and $q=2$ via the notations in Lemma \ref{estprep}.  Let $\omega$ be any irrational number.
 \begin{itemize}
 \item If $A_s=\mbox{\rm diag}(\omega, \ -2 \omega,\ 3, -1)$, then $I(A_s)$ is generated by $\varphi_1=x_1^2x_2$ and $\varphi_2=x_3x_4^3$, and $\mathcal C^{for}(A_s)$ consists of the elements $g(x)=\left(x_1g_1(\varphi_1,\varphi_2),\ x_2g_2(\varphi_1,\varphi_2),\ x_3g_3(\varphi_1,\varphi_2),\ x_4g_4(\varphi_1,\varphi_2)\right)^\tau$, with the $g_i$'s any formal series in two variables.
\item If  $A_s=\mbox{\rm diag}(\omega, \ 1,\ \omega+2, -2\omega-3)$, then $I(A_s)$ is generated by $\phi_1=x_1x_2x_3x_4$, $\phi_2=x_1x_3^3x_4^2$ and $\phi_3=x_1^2x_2^3x_4$, any two of which are functionally independent and the three are functionally dependent. This generator system is minimal, and there exists no algebraically independent generator system.

Now the elements in $\mathcal C^{for}(A_s)$ are in general of the form $g(x)=\hat g(x)+\tilde g(x)$ with $\hat g=\left(x_1g_1 ,\ x_2g_2 ,\ x_3g_3 ,\ x_4g_4 \right)^\tau$ and the $g_i$'s any formal series in the variables $\phi_1,\phi_2,\phi_3$, and $\tilde g(x)$ a series in resonant monomials whose $j$th component does not have a factor $x_j$.
\end{itemize}
}
\end{example}

The linear vector fields in the centralizer of a PDNF provide a lower bound for the centralizer dimension, as the next result shows.
\begin{theorem}\label{lowerest} Let the notation be as in Lemma \ref{estprep}, and let condition \eqref{fullfires} be satisfied. Then, for
given $f=A_s+\cdots$ in PDNF the subspace of linear vector fields in ${\mathcal C}^{\rm for}(f)$ has dimension $\geq q$, and $\dim {\mathcal C}^{\rm for}(f)\geq q+1$. Moreover, Z-generically the subspace of linear vector fields in the centralizer has dimension $q$.
\end{theorem}

\begin{proof}
With Lemma \ref{estprep}$(c)$ one sees that all $C_i\in  {\mathcal C}^{\rm for}(f)$ for any $f$ in PDNF. This shows the first assertion. The second assertion follows directly whenever $f\not=A_s$ by augmenting the linear centralizer elements with $f$, while in case $f=A_s$ the centralizer dimension is infinite by Lemma \ref{fireslem}.

We turn to the proof of the third assertion. Thus let $D \in  {\mathcal C}^{\rm for}(f)$  be linear. We first show that $D$ is necessarily diagonal in a Z-generic setting. If the eigenvalues of $A_s$ are pairwise different then this is obvious from Lemma \ref{lincommutelem}. Otherwise we impose the Z-generic condition that the rank of $A_n$ is maximal. With $D=D_s+D_n$ we have that $D_n$ commutes with $A_s$ and $A_n$, hence (assuming that equal eigenvalues of $A_s$ are listed consecutively as in \eqref{linsemi}) $D$ is block diagonal with strictly upper triangular blocks by Lemma \ref{lincommutelem} in the Appendix. We now assume $D_n\not=0$, and moreover we may assume that its upper left block is nontrivial. Thus there is some $r>1$ such that $\lambda_1=\cdots=\lambda_r$, and from \eqref{fullfires} we find that for any $j\in\{1,\ldots,r\}$ the monomial
\[
\phi:=x_j^{d_1+\cdots+d_r}x_{r+1}^{d_{r+1}}\cdots x_{n}^{d_{n}}
\]
satisfies $X_{A_s}(\phi)=0$, hence $\left[A_s,\phi\cdot A_s\right]=0$. Now one can choose $j$ such that
\[
X_{D_n}(x_j)=\sigma_{j+1}x_{j+1}+\cdots+\sigma_r x_r\not=0,
\]
which implies
\[
X_{D_n}(\phi)=(d_1+\cdots +d_r)x_j^{d_1+\cdots+d_r-1}X_{D_n}(x_j)\cdot x_{r+1}^{d_{r+1}}\cdots x_{n}^{d_{n}}+x_j^{d_1+\cdots+d_r}\cdot\left(\cdots\right)
\]
is nonzero: This is obvious if the second term is zero. Otherwise it has higher degree in $x_j$ than the first, hence cancellation is impossible. Therefore
\[
\left[ D_n,\phi\cdot A_s\right]=X_{D_n}(\phi)\cdot A_s\not=0.
\]
So, given the Z-generic condition (for vector fields of degree $\leq d_1+\cdots+d_n$) that $\phi\cdot A_s$ appears in the PDNF with nonzero coefficient, one sees that the upper left block of $D_n$ must be trivial, and repeating the argument shows that $D=D_s$ is diagonal. Furthermore, letting $\psi_i$ as in Lemma \ref{estprep} and assuming the Z-generic conditions that all $\psi_i\cdot A_s$ appear in the PDNF with nonzero coefficient, we see that all $X_D(\psi_i)=0$, whence $D$ is a $\mathbb C$-linear combination of the $C_i$. This completes the proof.
\end{proof}
\begin{corollary}
Let $h(x)=Bx+\sum_{j\geq 2} h_j(x)$ be a nonlinear formal vector field on $\mathbb C^n$ such that the eigenvalues of $B$ satisfy condition \eqref{fullfires}, and let $q$ be the dimension of the $\mathbb Q$-vector space spanned by the eigenvalues of $B$. Then  $\dim\, {\mathcal C}^{\rm for}(h)\geq q+1$.
\end{corollary}
%%%%%%%%%%%%%%%%%%%%%%%%%%%%%%%%%%%%%%%%%%%%%%%%
\subsection{An algebraically distinguished class}\label{subsec42}
As noted above, the structures of $I(A_s)$ and of $\mathcal{C}^{\rm pol}(A_s)$ may be quite complicated; in particular the minimal number of generators may become very large.
In the present subsection we will focus on normal forms $f=A+\cdots$ for which these structures are as simple as possible, from an algebraic perspective.
\begin{lemma}\label{distinglem}
Let $f=A+\cdots$ be in PDNF with the property that $A$ has no eigenvalue zero and the eigenvalues satisfy condition \eqref{trivresonly}, hence ${\mathcal C}^{\rm pol}(A_s)$ is a free module over $I(A_s)$. Moreover assume that $I(A_s)$ admits an algebraically independent set $\psi_1,\ldots,\psi_r$ of generators. Then:
\begin{enumerate}[$(a)$]
\item $A=A_s$ is semisimple with pairwise different eigenvalues, and the $\psi_i$ may be taken as monomials in the eigencoordinates of $A_s$.
\item There exist formal power series $\sigma_{i}$ in $r$ variables such that
\[
X_f(\psi_i)=\psi_i \sigma_{i}(\psi_1,\ldots,\psi_r),\quad \sigma_i(0)=0,\quad 1\leq i\leq r.
\]
\item One has the expansions
\[
\sigma_i =\sum_{j=1}^r \nu_{ij}\psi_j+\text{ h.o.t},\quad \nu_{ij}\in\mathbb C,
\]
and for every matrix $\left(\nu^*_{ij}\right)\in\mathbb C^{r\times r}$ there exists an $f=A+\cdots$ in normal form such that $\sigma_i=\sum_{j=1}^r \nu^*_{ij}\psi_j+\cdots$.
\end{enumerate}
\end{lemma}
\begin{proof}
We may assume that $A_s$ is diagonal. Then part $(a)$ follows from \eqref{trivresonly} and the proof of Lemma \ref{trivresonlylem}. To prove part $(b)$ we note that
\[
f=A+\sum_{k=1}^n\left(\sum_{j=1}^r \rho_{kj}\psi_j+\text{ t.h.o.}\right)Q_k
\]
with $Q_k(x)=x_ke_k$ and constants $\rho_{kj}$,
thus the assertion follows from
\[
X_{Q_j}(x_1^{m_1}\cdots x_n^{m_n})=m_j\cdot x_1^{m_1}\cdots x_n^{m_n},
\]
and $X_A(\psi_j)=0$. To prove part $(c)$, we write
\[
\Psi(x)=\begin{pmatrix}\psi_1(x)\\ \vdots \\ \psi_r(x)\end{pmatrix}
\]
and note that the Jacobian $D\Psi(x)$ generically has rank $r$, due to algebraic independence. Now there exist nonegative integers $m_{ik}$ such that
\[
X_{Q_k}(\psi_i)=m_{ik}\psi_i,
\]
and with the matrix $M=(m_{ik})\in\mathbb Z^{r\times n}$ we can rewrite these relations as
\[
D\Psi(x)\,{\rm diag}(x_1,\ldots,x_n)={\rm diag}(\psi_1(x),\ldots,\psi_r(x))\,M.
\]
Since the generic rank of $D\Psi(x)$ equals $r$ and both diagonal matrices are generically invertible, we find that ${\rm rank}\,M=r$. Now the expansion of $f$ above yields
\[
X_f(\psi_i)=\psi_i\,\left(\sum\nu_{ij}\psi_j+\text{ t.h.o. }\right),
\]
with
\[
\nu_{ij}=\sum_k m_{ik}\rho_{kj}.
\]
Since the matrix $M$ defines a surjective linear map, and the $\rho_{kj}$ can be chosen arbitrarily, any matrix $(\nu^*_{ij})\in\mathbb C^{r\times r}$ can be obtained in this way.
\end{proof}
In view of Lemma \ref{distinglem}, Proposition \ref{nofofinprop} and Remark \ref{rigidrem} we now consider formal vector fields
\begin{equation}\label{redvf}
\widehat f(y)=\sum_{k\geq 2} \widehat f_k(y) \in\mathbb C[[y_1,\ldots,y_r]]^r,
\end{equation}
with quadratic part
\begin{equation}\label{homquadvf}
\widehat f_2(y)=\begin{pmatrix}y_1\left(\sum_j\nu_{1j}y_j\right)\\ \vdots\\ y_r\left(\sum_j\nu_{rj}y_j\right)\end{pmatrix}.
\end{equation}
For these vector fields, {generically} the centralizer of $\widehat f$ is trivial, i.e.\ equal to $\mathbb C\widehat f$.
\begin{theorem}\label{quadcentthm}
Let the vector field $\widehat f$ be given by \eqref{redvf} with \eqref{homquadvf},
and define $\widehat{\mathcal R}\subseteq\mathbb C^{r\times r}$ by the property that $\left(\nu_{ij}\right)\in\widehat{\mathcal R}$ if and only if the formal centralizer of $\widehat f_2+{\rm h.o.t.}$ is trivial for any choice of higher order terms. Then:
\begin{enumerate}[$(a)$]
\item $\widehat{\mathcal R}$ has full measure, i.e., its complement in $\mathbb C^{r\times r}$ has measure zero.
\item $\widehat{\mathcal R}$ contains a nonempty open subset of $\mathbb C^{r\times r}$.
\end{enumerate}
\end{theorem}
\begin{proof} \begin{enumerate}[$(i)$]
\item We invoke Lemma \ref{adelemcent} from the Appendix:
Assume there exists $c$ such that $\widehat f_2(c)=c\not=0$, with eigenvalues $\mu_1,\ldots,\mu_r$ of $D\widehat f_2(c)$. If $q\not=0$ is homogeneous of degree $s>1$ such that $[\widehat f_2,q]=0$ then there exist nonnegative integeres $\ell,\ell_1,\ldots,\ell_r$ and some $k\in\{1,\ldots,r\}$ such that
\[
\sum\ell_i+\ell=s\text{ and }\sum\ell_i\,\mu_i+\ell=\mu_k.
\]
This observation may yield degree bounds for homogeneous vector fields commuting with $\widehat f_2$. Specifically, if $\mu_1,\ldots,\mu_r$ are linearly independent over the rationals $\mathbb Q$ then necessarily $s=2$, since the second relation in \eqref{quadraticcent} for $k>1$ would imply $\ell_k=1$, all other $\ell_j=0$, and $\ell=0$ (a contradiction to $s>1$), and for $k=1$ one gets $\ell_2=\cdots=\ell_r=0$ and $2\ell_1+\ell=2,\,\ell_1+\ell=s$, which leaves only the possibility that $\ell_1=0$ and $\ell=2$. Moreover, the proof of \cite{WaADE}, Proposition 10.5 shows that $q$ is a scalar multiple of $\widehat f_2$.
\item Generally, if the only homogeneous vector fields commuting with $\widehat  f_2$ are the scalar multiples of $\widehat f_2$, then the formal centralizer of $\widehat f$ is trivial: If $\widehat g=\widehat g_s+\cdots\not=0 $ commutes with $\widehat f$ then $[\widehat f_2,\,\widehat g_s]=0$, hence $s=2$ and $\widehat g_2=\beta \widehat f_2$
 for some scalar $\beta$. Now $\widehat g-\beta\widehat f$ commutes with $\widehat f$ and must be equal to zero since it cannot have lowest order term of degree $>2$. Hence $\widehat g-\beta\widehat f=0$.
\item
We may assume that $\nu_{11}\not=0$ as one defining condition for $\widehat{\mathcal R}$. Then
\[
c:=\begin{pmatrix} \nu_{11}^{-1}\\ 0\\ \vdots\\0\end{pmatrix}
\]
satisfies $\widehat f_2(c)=c$. With
\[
D\widehat f_2(y)\,z=\begin{pmatrix}y_1\left(\sum_j\nu_{1j}z_j\right)+z_1\left(\sum_j\nu_{1j}y_j\right)\\ \vdots\\ y_r\left(\sum_j\nu_{rj}z_j\right)+z_r\left(\sum_j\nu_{rj}y_j\right)\end{pmatrix}
\]
one first notes that $e_1$ is an eigenvector with eigenvalue $2$, and for $2\leq j\leq r$ one has
\[
D\widehat f_2(c)e_j=\nu_{11}^{-1}\nu_{1j}e_1+\nu_{j1}\nu_{11}^{-1} e_j.
\]
Therefore the matrix representing $D\widehat f_2(c)$ is upper triangular, with eigenvalues $2$ and $\nu_{j1}\nu_{11}^{-1}$, $2\leq j\leq r$.
By parts $(i)$ and $(ii)$ one sees that ${\mathcal C}_0^{\rm for}(\widehat f_2+{\rm h.o.t.})$ is trivial whenever $\nu_{11},\ldots,\nu_{r1}$ are linearly independent over the rationals $\mathbb Q$ (and the remaining $\nu_{jk}$ are arbitrary). Since this property defines the complement of a measure zero set, statement $(a)$ is proven.
\item
For statement $(b)$ we first consider the special vector field
\[
P(y)=\begin{pmatrix}y_1^2\\ 2y_1y_2\\ \vdots\\2y_1y_r\end{pmatrix},
\]
with $c:=e_1$ satisfying $P(c)=c$ and $D{P}(c)=2 I_r$. For all $(\varepsilon_{ij})\in\mathbb C^{r\times r}$ with $|\varepsilon_{ij}|$ sufficiently small the vector field
\[
\widetilde P(y):=P(y)+\begin{pmatrix}y_1\left(\sum_j\varepsilon_{1j}y_j\right)\\ \vdots\\ y_r\left(\sum_j\varepsilon_{rj}y_j\right)\end{pmatrix}
\]
satisfies $\widetilde P(c)=(1+\varepsilon_{11}) c$, hence $\widetilde P(\alpha c)=\alpha c$ for some $\alpha$ near $1$, and all eigenvalues $\rho_i$ of $D\widetilde P(\alpha e_1)$ satisfy
\[
1\leq a<{\rm Re}\,(\rho_i)\leq A<3
\]
with suitable real $a$ and $A$.\\
Now we assume that there exists a vector field $Q$, homogeneous of degree $s\geq 2$, that commutes with $\widetilde P$. By \eqref{quadraticcent} in Lemma \ref{adelemcent} there exist nonnegative integers $\ell_i$ and $\ell$ such that $\sum \ell_i +\ell =s$ and $\sum\ell_i\rho_i+\ell=\rho_k$ for some $k$. We obtain the estimates
\[
a\cdot s\leq \sum\ell_i\,{\rm Re}\,\rho_i+\ell={\rm Re}\,\rho_k<A
\]
and therefore $s<A/a<3$, which implies $s=2$ by part $(i)$ of the proof. So, for an open set in the coefficient space of all homogeneous quadratic vector fields we have that commuting homogeneous vector fields must have degree two.
\item  The set of all vector fields $\widehat f_2$ as in \eqref{homquadvf} admitting a nontrivial homogeneous commuting vector field of degree two is Zariski-closed: Evaluating $[\widehat f_2,\widehat g_2]=0$, with undetermined coefficients for $\widehat g_2$, one obtains a linear system of equations with matrix entries depending polynomially on the $\nu_{ij}$. This system has only the trivial solutions (corresponding to $\widehat g\in\mathbb C\cdot \widehat f_2$) for some {$\widehat f_2$},  as was noted in part $(i)$. Hence the determinant conditions necessary for the existence of nontrivial solutions define a proper (Zariski) closed set in the space of all homogeneous quadratic vector fields. To summarize, {by part $(ii)$} we have trivial centralizer for all vector fields $\widehat f_2+\cdots$, with the $(\nu_{ij})$ in a nonempty and open subset of $\mathbb C^{r\times r}$.
\end{enumerate}
\end{proof}
Next we show that generically the centralizer of a vector field $f$ in PDNF is spanned by $f$ and linear vector fields in the distinguished algebraic setting under consideration.
\begin{theorem}\label{disalgcentthm} Let
\[
f(x)=Ax+\sum_{j\geq 2}f_j(x)
\]
be in PDNF, and assume that the hypotheses of Lemma \ref{distinglem} are satisfied. Let $L:= {\rm max}\,\{{\rm deg}\,\psi_i,\,1\leq i\leq r\}$  and let $W$ be the space of coefficients of $f_2+\cdots+f_L$. Define ${\mathcal R}\subseteq W$ by the property that the coefficients of $f_2+\cdots +f_L$ lie in ${\mathcal R}$ if and only if the formal centralizer of $f=A+f_2+\cdots f_L+{\rm h.o.t.}$ contains only linear combinations of $f$ and linear vector fields, for any choice of higher order terms. Then:
\begin{enumerate}[$(a)$]
\item ${\mathcal R}$ has full measure in $W$.
\item ${\mathcal R}$ contains a nonempty open subset of $W$.
\end{enumerate}
\end{theorem}
\begin{proof}
\begin{enumerate}[$(i)$]
\item With the module generators $Q_i$ we write
\[
f(x)=Ax+\sum_{1\leq i\leq n}\eta_i(x)Q_i(x);\quad \eta_i(x)=\widehat\eta_i(\psi_1(x),\ldots,\psi_r(x))
\]
with uniquely determined $\widehat\eta_i\in\mathbb C[[y_1,\ldots,y_r]]$. Expand
\[
\eta_i(y)=\sum\eta_{ij}y_j+\text{ t.h.o.},
\]
then the reduced vector field of $f$ has the form \eqref{redvf} with quadratic part \eqref{homquadvf}, and by definition of $L$ the reduced vector field of $A+f_2+\cdots+f_L$ has the same quadratic part. Now define $\mathcal{R^*}\subseteq W$ by the property that an element of $W$ lies in $ \mathcal{R^*}$ if and only if the corresponding coefficients of $ \widehat f_2$ lie in $\widehat {\mathcal R}\subseteq \mathbb C^{r\times r}$, as defined in Theorem \ref{quadcentthm}. % {\color{blue} Then ${\color{blue} \mathcal{R^*}}=\mathcal R$.}
 Since the mapping from coefficients of $f_2+\cdots+f_L$ to coefficients of $\widehat f_2$ is linear and surjective, we see that the complement of  $\mathcal{R^*}$ has measure zero, and $\mathcal{R^*}$ contains a nonempty open set. {We will now show that ${\mathcal R}^*\subseteq\mathcal R$, which implies the assertion of the theorem. For the following assume that the coefficients of $f_2+\cdots+f_L$ lie in ${\mathcal R}^*$.}
\item
 From Theorem \ref{quadcentthm} {and the definition of ${\mathcal R}^*$} we see that the centralizer of $\widehat f$ is trivial. Thus, if $g$ commutes with $f$ then we may assume $\widehat g=0$, hence
\[
g(x)=\sum_{1\leq \ell\leq n}\theta_\ell(x) Q_\ell(x),\quad \theta_\ell(x)=\widehat\theta_\ell(\psi_1(x),\ldots,\psi_r(x))
\]
with uniquely determined $\widehat\theta_\ell\in\mathbb C[[y_1,\ldots,y_r]]$, and furthermore
\[
X_g(\psi_i)=0 \text{ for }1\leq i\leq r.
\]
\item We next evaluate the commutator relation {$[f,\,g]=0$}. For all $i,\,j\in\{1,\ldots,n\}$ one has (by standard properties of the Lie bracket)
\[
\left[\eta_iQ_i,\theta_j Q_j\right]=\eta_iX_{Q_i}(\theta_j)Q_j-\theta_j X_{Q_j}(\eta_i)Q_i +\eta_i\theta_j[Q_i,\,Q_j]
\]
Since the last term always vanishes and $[A,\,g]=0$, summation yields
\[
0=[f,\,g]=\sum_{i,j} \eta_iX_{Q_i}(\theta_j)Q_j -\sum_{i,j}\theta_jX_{Q_j}(\eta_i)Q_i
\]
The second term yields
\[
\sum_{i=1}^n\sum_{j=1}^n \theta_jX_{Q_j}(\eta_i)Q_i=\sum_{i=1}^n X_g(\eta_i)\,Q_i=0,
\]
since $\eta_i=\widehat\eta_i(\psi_1,\ldots,\psi_r)$ and every $\psi_j$ is a first integral of $g$. There remains
\[
0=\sum_{j=1}^n\sum_{i=1}^n \eta_iX_{Q_i}(\theta_j)Q_j =\sum_{j=1}^n X_f(\theta_j)Q_j.
\]
Since the $Q_j$'s form a free system, this is equivalent to
\[
X_f(\theta_j)=0,\quad 1\leq j\leq n;
\]
in other words, each $\theta_j$ is a first integral of $f$.
\item Passing to the reduced vector field, we have that each $\widehat \theta_j$ is a first integral of $\widehat f$. {By our assumption, the coefficients of $\widehat f_2$ lie in $\widehat{\mathcal R}$, and we will show that the only first integrals of $\widehat f_2$ are the constants.} Therefore $g$ is necessarily linear, and the theorem is proven.\\
Thus assume that $\widehat\theta=\widehat\theta_s+\cdots\in\mathbb C[[y_1,\ldots,y_r]]$ is a nonconstant first integral of $\widehat f$, with $s>0$ and $\widehat \theta_s\not=0$. Then $\widehat \theta_s$ is a homogeneous polynomial first integral of $\widehat f_2$. As in the proof of Lemma \ref{quadcentthm} we consider $c=\begin{pmatrix} \ast\\ 0\\ \vdots\\ 0\end{pmatrix}$ with $\widehat f_2(c)=c\not=0$, and let $\mu_1=2,\mu_2,\ldots,\mu_r$ be the eigenvalues of $D\widehat f_2(c)$. According to Lemma \ref{adelem} from the Appendix (in the special case $\lambda=0$)  there exist nonnegative integers $k\leq s$, $k_1,\ldots, k_{r}$ such that
\[
\sum k_i+k=s,\quad \sum k_i\mu_i+k=0
\]
But such a relation cannot hold when the $\mu_i$'s are linearly independent over $\mathbb Q$, and neither can it hold when all $\mu_i$'s have positive real parts. By the definition of $\widehat{\mathcal R}$  the proof is finished.
\end{enumerate}
\end{proof}
Passing to normal forms and back we obtain the following result. For the proof note only that the map sending the coefficients of a Taylor polynomial to the coefficients of the Taylor polynomial of its PDNF (made unique in a suitable manner) is rational and surjective.
\begin{corollary}
Let $h(x)=Bx+\sum_{j\geq 2} h_j(x)$ be a nonlinear formal vector field on $\mathbb C^n$ such that $B=B_s$ satisfies the hypotheses of Lemma \ref{distinglem}. Then there exists an $L>0$ such that the coefficient space $\widetilde W$ of $h_2+\cdots+h_L$ contains a full measure subset $\widetilde {\mathcal R}$ with the property: For coefficients in $\widetilde {\mathcal R}$ the formal centralizer of $h$ is uniquely determined by the formal centralizer of $B+h_2+\cdots+h_L$, and has finite dimension.
\end{corollary}
%%%%%%%%%%%%%%%%%%%%%%%%%%%%%%%%%%%%%%%%%%%%%%%%%%%%%%%%%%%%%%
\subsection{Some classes of examples}
While the distinguished algebraic setting discussed in the previous subsection poses strong restrictions on the linear parts, we now exhibit some examples to show that this class is reasonably large. We start with the case when the invariant algebra of $A_s$ has a single generator; for this a complete analysis is possible. The result in part $(b)$ of the following Proposition is known from \cite{Wa00b}, Example 2, where a different proof was given.
\begin{proposition}\label{singleprop} In the setting of Lemma \ref{estprep}, let $q=n-1$, and assume that \eqref{fullfires} is satisfied, w.l.o.g.\ with the $d_i$ relatively prime. Set $\psi:=x_1^{d_1}\cdots x_n^{d_n}$.
\begin{enumerate}[$(a)$]
\item Then the $\lambda_i$ are pairwise different, hence $A_n=0$, and every PDNF is of the form
\[
f(x)=A_sx+\sum_{j\geq 1} \psi(x)^jU_jx
\]
with diagonal matrices $U_j$.
\item Whenever some $X_{U_j}(\psi)\not=0$ then ${\mathcal C}^{\rm for}(f)$ is spanned by $f$ and the linear vector fields $C_1,\dots,C_{n-1}$  from Lemma \ref{estprep}, thus $\dim {\mathcal C}^{\rm for}(f)= n$.
\item {If all $X_{U_i}(\psi)=0$ then, with $C_1,\ldots,C_{n-1}$ as in Lemma \ref{estprep} we have
\[
f(x)=\sigma_1(\psi(x))C_1x+\cdots +\sigma_{n-1}(\psi(x))C_{n-1}x
\]
with series $\sigma_i$ in one variable. If all $\sigma_i$'s are constants then $f=A$. If some $\sigma_j$ is not constant then
\[
{\mathcal C}^{\rm for}(f)=\left\{\rho_1(\psi(x))C_1x+\cdots +\rho_{n-1}(\psi(x))C_{n-1}x; \,\rho_i\in\mathbb C[[x_1]]\right\}.
\]
}
\end{enumerate}
\end{proposition}

\begin{proof} For part $(a)$ see Lemma \ref{distinglem}. We turn to statements $(b)$ and $(c)$.
\begin{enumerate}[$(i)$]
\item { Assume that some  $X_{U_j}(\psi)\not=0$, and let $g\in{\mathcal C}^{\rm for}(f)\subseteq {\mathcal C}^{\rm for}(A_s)$ and therefore
\[
g(x)= \sum_{k\geq 0} \psi(x)^kB_kx
\]
with diagonal matrices $B_k$.}
By Corollary \ref{reducomm} we obtain a symmetry reduction for $f$ to dimension one, in the form
\[
X_f(\psi)=\widehat f(\psi):=\sum\psi(x)^{j+1}\theta_j \not=0,
\]
with constants $\theta_j$ such that  $X_{U_j}(\psi)=\theta_j\psi$, and likewise we get
\[
X_g(\psi)=\widehat g(\psi):=\sum\psi(x)^{j+1}\sigma_j
\]
with constants $\sigma_j$ such that $X_{B_j}(\psi)=\sigma_j\psi$. The argument in the proof of Theorem \ref{disalgcentthm} shows that we may take $X_g(\psi)=0$, thus all $X_{B_j}(\psi)=0$.
\item Now with $C_1,\ldots,C_{n-1}$ as in Lemma \ref{estprep} we may write
\[
g(x)=\tau_1(\psi(x))C_1x+\cdots +\tau_{n-1}(\psi(x))C_{n-1}x
\]
where the $\tau_i$ are series in one variable, and obtain
\[
\left[f,\,g\right]=\sum X_f(\tau_j)C_j
\]
since all $\left[f,\,C_j\right]=0$. Finally this implies that all $X_f(\tau_j)=0$ by Lemma \ref{estprep}$(a)$, hence all $\tau_j$ are constants{, and part $(b)$ is proven.}
\item {The representation of $f$ follows with Lemma \ref{estprep}, moreover for $g\in{\mathcal C}^{\rm for}(f)$ we have a representation
\[
g(x)=\tau_1(\psi(x))C_1x+\cdots +\tau_{n-1}(\psi(x))C_{n-1}x+\tau_n(\psi(x))Ix
\]
with the identity matrix $I$. Then
\[
[g(x),\,f(x)]=[\tau_n(\psi(x))Ix,f(x)]=-\tau_n(\psi(x))\sum X_I(\sigma_j(\psi(x))C_jx,
\]
and $\tau_n\not=0$ implies that all $\sigma_j$ are constant, as asserted in part $(c)$.
}
\end{enumerate}
\end{proof}
We present two further classes of linear vector fields $B$ that satisfy condition \eqref{trivresonly} and have $I(B)$ admitting an algebraically independent generator system.
{
\begin{proposition}\label{maxindprop}
Let $r>1$ and ${0=m_0<}m_1<\cdots<m_{r-1}<m_r=n$, moreover let $d_1,\ldots,d_n$ be positive integers, and set
\[
D_i:={\sum_{k=m_{i-1}+1}^{m_{i}}}d_ke_k^\tau,\quad 1\leq i\leq r.
\]
Assume furthermore that {$d_{m_{i-1}+1},\ldots,d_{m_{i}}$} are relatively prime for each $i\in\{1,\ldots,r\}$. Now let $\mu_1,\ldots,\mu_n\in\mathbb C$, $\mu=(\mu_1,\ldots,\mu_n)$ such that
\[
\left<D_1,\mu\right>=\cdots=\left<{D_r},\mu\right>=0;\quad \dim_{\mathbb Q}\left(\mathbb Q\mu_1+\cdots+\mathbb Q\mu_n\right)=n-r.
\]
Then the following hold for $B:={\rm diag}\,(\mu_1,\ldots,\mu_n)$.
\begin{enumerate}[$(a)$]
\item  $I(B)$ is generated by the algebraically independent monomials $\phi_i(x):=x^{D_i}$, $1\leq i\leq r$.
\item The module ${\mathcal C}^{\rm pol}(B)$ is generated by the $Q_j(x)=x_je_j$, $1\leq j\leq n$, and free.
\end{enumerate}
\end{proposition}
\begin{proof}
As a vector space over $\mathbb C$, $I(B)$ is spanned by all monomials $x^\ell$ with $\left<\ell,\,\mu \right>=0$. By construction of $B$, every such $\ell$ is a $\mathbb Q$-linear combination of the $D_i$. By positivity the coefficients in this linear combination must be nonnegative, and by relative primeness of the $d_{m_{i}},\ldots,d_{m_{i+1}}$ the coefficients must be integers. This proves part $(a)$. For the proof of part $(b)$, note first that the $\mu_i$'s are nonzero and pairwise different due to the dimension requirement, hence every linear vector field commuting with $B$ is a linear combination of the $Q_j$. Furthermore recall that a vector monomial $x^Ne_k^\tau$, $|N|\geq 2$ commutes with $B$ if and only if
\[
\left<N-e_k^\tau,\mu\right>=0.
\]
If $i$ is such that $m_{i}<k\leq m_{i+1}$ then the coefficient of $D_i$ in the linear combination must be positive, thus $x^Ne_k^\tau=\psi\cdot x_ke_k^\tau$ with some $\psi\in I(B)$.
\end{proof}
}
\begin{example}{\em
The class characterized in the above proposition includes the nonresonant coupled oscillators, with
\[
B={\rm diag}\left(i\omega_1,-i\omega_1,\ldots,i\omega_m,-i\omega_m\right)
\]
and $\omega_1,\ldots,\omega_m\in\mathbb R_+$ linearly independent over the rationals.
}
\end{example}
{ For a further class we explicitly construct $B$ with integer eigenvalues.
\begin{proposition}\label{nonindprop}
Let $\ell_1,\ldots,\ell_n$ be pairwise relatively prime integers, all $\ell_i>1$, and $L:=\ell_1\cdots\ell_n$. Moreover let $\varepsilon_1,\ldots,\varepsilon_n\in\{1,\,-1\}$, and set
\[
B:={\rm diag}\,\left(\varepsilon_1L/\ell_1,\ldots,\varepsilon_nL/\ell_n\right).
\]
Then the following hold.
\begin{enumerate}[$(a)$]
\item The $I(B)$-module ${\mathcal C}^{\rm pol}(B)$ is free, and generated by the $Q_j(x)=x_je_j$, $1\leq j\leq n$.
\item If $($w.l.o.g.$)$ $\varepsilon_n=-1$ and all other $\varepsilon_i=1$ then $I(B$) admits an algebraically independent generator set, viz.
\[
\phi_1(x)={x_1^{\ell_1}x_n^{\ell_n}},\ldots, \phi_{n-1}(x)={x_{n-1}^{\ell_{n-1}}x_n^{\ell_{n}}}.
\]
\end{enumerate}
\end{proposition}
\begin{proof} To prove part $(a)$, assume (for instance) that there are nonnegative integers $m_2,\ldots,m_n$ such that
\[
\varepsilon_2m_2L/\ell_2+\cdots+\varepsilon_nm_nL/\ell_n=L/\ell_1.
\]
Rewriting and setting $L^*:=\ell_2\cdots\ell_n$ one gets
\[
\ell_1\left(\varepsilon_2m_2L^*/\ell_2+\cdots+\varepsilon_nm_nL^*/\ell_n\right)=L^*,
\]
hence $\ell_1$ divides $L^*$; a contradiction.\\
 To prove part $(b)$, let $d_1,\ldots,d_n$ be nonnegative integers, not all zero, such that
\[
d_1L/\ell_1+\cdots + d_{n-1}L/\ell_{n-1}-d_nL/\ell_n=0.
\]
Then necessarily $d_i=\ell_id_i^*$ for all $i$, and there remains
\[
d_1^*+\cdots+d_{n-1}^*-d_n^*=0.
\]
From this one sees that $(d_1^*,\ldots,d_n^*)$ is a nonnegative integer combination of $e_1^\tau+e_n^\tau,\ldots,e_{n-1}^\tau+e_n^\tau$, and the assertion follows.
\end{proof}
\begin{remark} {\em Assume the setting of Proposition \ref{nonindprop} with dimension $n>3$. Then the algebra $I(B)$ does not admit an algebraically independent generator set whenever two of the $\varepsilon_i$ are positive and two negative. This follows from considering the last step in the proof above: For instance when $n=4$ and
\[
1=\varepsilon_1=\varepsilon_2=-\varepsilon_3=-\varepsilon_4,
\]
one arrives at
\[
d_1^*+d_2^*-d_3^*-d_4^*=0,
\]
and to obtain all nonnegative solutions one needs the four generators $e_1^\tau+e_3^\tau$, $e_1^\tau+e_4^\tau$, $e_2^\tau+e_3^\tau$, $e_2^\tau+e_4^\tau$.
}
\end{remark}
}

Finally we characterize the three dimensional linear vector fields with no eigenvalue zero for which the algebraically distinguished setting holds. If the eigenvalues span a two dimensional vector space over the rationals then Proposition \ref{singleprop} applies. There remains the case when the eigenvalues span a one dimensional vector space over the rationals. The following result covers all cases, up to a time scaling.
\begin{proposition}\label{dim3prop}
Let $d_1,\,d_2,\,d_3$ be positive integers such that their greatest common divisor satisfies ${\rm gcd}\,(d_1,d_2,d_3)=1$, and let $B={\rm diag}\,(d_1,d_2,-d_3)$. Then the following are equivalent:
\begin{itemize}
\item[$(a)$]  The module ${\mathcal C}^{\rm for}(B)$ is generated by $Q_1,Q_2,Q_3$ and $I(B)$ admits an algebraically independent generator system.
\item[$(b)$] There exist relatively prime $\ell_1>1$ and $\ell_2>1$ such that
\[
\ell_2\text{ divides } d_1,\,\ell_1\text{ divides } d_2\text{  and  } d_3=\ell_1\ell_2.
\]
\end{itemize}
\end{proposition}
\begin{proof}{The work-intensive part is the implication from $(a)$ to $(b)$. This will be done first, in parts $(i)$ through $(iv)$.}
\begin{enumerate}[$(i)$]
\item We first discuss the conditions on the module generators. The first condition states that the equation
\begin{equation}\label{auxeq3d}
d_2m_2-d_3m_3=d_1
\end{equation}
has no solution in nonnegative integers $m_2,\,m_3$. But if $d_2$ and $d_3$ are relatively prime then such a solution exists: One has a relation $d_2s_2+d_3s_3=d_1$ with integers $s_2,\,s_3$, and more generally $d_2(s_2+k\cdot d_3)+d_3(s_3-k\cdot d_2)=d_1$, with any integer $k$. Now a suitable choice of $k$ shows the existence of positive $m_2$ and $m_3$ { such that \eqref{auxeq3d} holds}, a contradiction.
Conclusion: There is an integer $\ell_1>1$ such that $d_2=\ell_1 d_2^*$ and $d_3=\ell_1 \widetilde d_3$, with ${\rm gcd}\,(d_2^*,\widetilde d_3)=1$. Conversely, if this holds then \eqref{auxeq3d} has no solution in nonnegative integers, due to  ${\rm gcd}\,(d_1,d_2,d_3)=1$.\\
By analogous arguments, the condition that
\[
d_1m_1-d_3m_3=d_2
\]
has no solution in nonnegative integers $m_1, m_3$ implies the existence of an integer $\ell_2>1$ such that $d_1=\ell_2 d_1^*$ and $d_3=\ell_2 \widehat d_3$, with ${\rm gcd}\,(d_1^*,\widehat d_3)=1$. Conversely, if this holds then \eqref{auxeq3d} has no solution in nonnegative integers. From ${\rm gcd}\,(d_1,d_2,d_3)=1$ one has ${\rm gcd}\,(\ell_1,\ell_2)=1$, hence
\[
d_3=\ell_1\ell_2d_3^*.
\]
The third condition that $d_1m_1+d_2m_2=-d_3$
has no solution in nonnegative integers is trivially satisfied. We thus have obtained necessary and sufficient conditions for the module generator property.
\item It is sufficient to consider monomials in $I(B)$, thus nonnegative integer solutions of
\[
0=n_1d_1+n_2d_2-n_3d_3=n_1\ell_2d_1^*+n_2\ell_1d_2^*-n_3\ell_1\ell_2d_3^*.
\]
By relative primeness, we see that $n_1=\ell_1m_1$, $n_2=\ell_2m_2$ with integers $m_1,\,m_2$. Setting $m_3:=n_3$, the above relation is equivalent to
\begin{equation}\label{auxeqtwo3d}
m_1d_1^*+m_2d_2^*-m_3d_3^*=0.
\end{equation}
\item The invarant algebra $I(B)$ admits an algebraically independent generator system if and only if there are two row vectors $M_1,\,M_2$  in $\mathbb Z_+^3$ such that every solution of \eqref{auxeqtwo3d} is a nonnegative integer linear combination of these two. Now there are two distinguished elements of $I(B)$ which correspond to
\[
(d_3^*,\,0,\,d_1^*)\text{  and  } (0,\,d_3^*,d_2^*).
\]
By relative primeness, every solution of \eqref{auxeqtwo3d} with a zero entry is an integer multiple of one of these. Moreover, each is a nonnegative integer linear combination of $M_1$ and $M_2$, which shows that neither $M_i$ can have all entries positive. Conclusion: Up to relabeling, $M_1=(d_3^*,0,d_1^*)$ and $M_2=(0,d_3^*,d_2^*)$.
\item Every nonnegative integer solution of \eqref{auxeqtwo3d} can uniquely be written as
\begin{equation}\label{auxeqthree3d}
\frac{r_1}{s}M_1+\frac{r_2}{s}M_2,
\end{equation}
with ${\rm gcd}\,(r_1,r_2,s)=1$ and $s$ necessarily dividing $d_3^*$. We show that $d_3^*>1$ implies the existence of a solution with $s>1$. Note that in such cases $M_1$ and $M_2$ cannot correspond to an algebraically independent generator system for $I(B)$.\\
Assuming $d_3^*>1$, take $s>1$ as a divisor of $d_3^*$. Then the third entry of the linear combination \eqref{auxeqthree3d} is an integer if and only if
\[
r_1d_1^*+r_2d_2^*\in s\mathbb Z.
\]
We show the existence of $r_1,\,r_2$ satisfying this and ${\rm gcd}\,(r_1,r_2,s)=1$. First, by relative primeness of $d_1^*,\,d_3^*$ there exist positive integers $a,\,b$ such that $as-bd_1^*=1$ (compare the argument in part $(i)$). Multiply by $d_2^*$ and rearrange to obtain
\[
(bd_2^*)\cdot d_1^*+1\cdot d_2^*=(ad_2^*)\cdot s.
\]
Conclusion: In order to admit an algebraically independent generator system for $I(B)$, one needs $d_3^*=1$. {The proof of $(a)\Rightarrow(b)$ is completed.}
\item { The proof of the reverse implication is straightforward: Assume that the conditions in $(b)$ hold, thus $d_1=\ell_2d_1^*$, $d_2=\ell_1d_2^*$. Then $\ell_1$ and $d_1^*$ are relatively prime due to ${\rm gcd}\,(d_1,d_2,d_3)=1$, and likewise $\ell_2$ and $d_2^*$ are relatively prime. \\
A relation of the form
\[
n_2d_2-n_3d_3=d_1\Leftrightarrow n_2\ell_1d_2^*-n_3\ell_1\ell_2=\ell_2d_1^*
\]
with nonnegative integers $n_2,\,n_3$ cannot hold, since $\ell_1$ does not divide the right-hand side. By this and analogous arguments one finds that the $Q_i$ generate the module. Moreover, the exponents of a monomial first integral of $B$ satisfy a relation
\[
n_1\ell_2d_1^*+n_2\ell_1 d_2^*-n_3\ell_1\ell_2=0\Leftrightarrow m_1d_1^*+m_2d_2^*-m_3=0
\]
with $n_1=\ell_1m_1,\, n_2=m_2\ell_2$ and $n_3=m_3$. But this implies
\[
(m_1,\,m_2,\,m_3)=m_1\cdot(1,\,0,\,d_1^*)+m_2\cdot(0,\,1,\,d_2^*),
\]
therefore $I(B)$ is generated by $x_1^{d_3}x_3^{d_1}$ and $x_2^{d_3}x_3^{d_2}$.
}
\end{enumerate}
\end{proof}
\subsection{Outlook}
We close this section with a few remarks on related work and open problems.
\begin{itemize}
\item As noted in the Introduction, Cerveau and Lins Neto \cite{CLN} obtained a rather complete description of the centralizer for local analytic vector fields in dimension two. {(In dimension two, blow-ups can be employed to reduce the problem to non-degenerate stationary points.)} In particular their results imply that the centralizer is generically trivial for stationary points with nilpotent linearization.
\item We discussed the analytic setting (involving convergence questions for PDNF) only in some special cases. The relevance of commuting vector fields and first integrals in convergence questions is well understood; see e.g. Stolovitch \cite{Sto} and Zung \cite{Zu02}.
\item In the formal (in dimension $>2$ also in the analytic) case the principal open question is concerned with triviality of the centralizer of a generic vector field with nilpotent linear part. The tools provided by Lemma \ref{adelemcent} and related results in \cite{WaADE} strongly depend on the special form of quadratic (more generally, of homogeneous polynomial) vector fields, and there seems to be no obvious generalization.
\item Whenever minimal generator systems of the algebra $I(A_s)$ become large, a direct application of the reduction property (see Proposition \ref{nofofinprop}) is no longer feasible. Note that minimal generator systems may be arbitrarily large even in dimension three, as shown by the example $A_s={\rm diag}\,(-12q,\,3,\,1)$, with $q>1$: Any generator system must contain the $4q+1$ monomials corresponding to the integer rows
\[
(1,0,12q),\,(1,1,12q-3),\, (1,2,12q-6),\,\ldots,(1, 4q-1,3),\,(1,4q,0).
\]
\item In the setting of Section \ref{infidimsec} the structure of the normalizer is not completely understood.

\end{itemize}
%%%%%%%%%%%%%%%%%%%%%%%%%%%%%%%%%%%%%%%%%%%%%%%%%%%%%%%%%%%%%%%%%%%%%%
%%%%%%%%%%%%%%%%%%%%%%%%%%%%%%%%%%%%%%%%%%%%%%%%%%%%%%%
\section{Jacobi multipliers}\label{jacobisec}
In view of their connection to normalizers, it is natural to include a discussion of Jacobi multipliers in the present paper. We first recall the definitions and some facts, and then show some general properties for PDNF. After discussing some particular examples, we focus on a distinguished setting for which it is shown that generically there exist no formal inverse Jacobi multipliers in dimension 3 and higher. In particular we obtain a complete picture for three dimensional vector fields that satisfy the consitions posed in subsection \ref{subsec42}.
\subsection{General properties}
We recall:
\begin{definition}\label{jacobi}{\em
An analytic function $\psi\not=0$ on an open connected subset $U^*\subseteq U$ is called an {\em inverse Jacobi multiplier $($or inverse Jacobi last multiplier$)$} of \eqref{deq} if
\begin{equation}\label{jaccon}
X_f(\phi)={\rm div}\,f\cdot \phi
\end{equation}
holds on $U^*$, with ${\rm div}\,f={\rm tr}\,Df$.\\
Mutatis mutandis, the definition carries over to local analytic and formal vector fields.
}
\end{definition}
In the situation of this definition, $\phi^{-1}$  is then called a Jacobi (last) multiplier. In the planar case Jacobi multipliers are usually called integrating factors. For more information on Jacobi multipliers, see \cite{WeZh16} and the paper by Berrone and Giacomini \cite{BG} which includes a survey of known facts, some of which we list here (noting that they carry over to the local and formal settings):
\begin{remark}\label{jacobibasics}{\em
\begin{enumerate}[$(a)$]
\item The zero set of an inverse Jacobi multiplier is an invariant set of \eqref{deq}.
\item The quotient of two Jacobi multipliers is a first integral of \eqref{deq}.
\item If system \eqref{deq} admits $n-2$ independent first integrals $\theta_1,\ldots,\theta_{n-2}$ and a Jacobi multiplier  $\phi$ on $U^*\subseteq U$ then a further first integral may be constructed from these by quadrature; thus $f$ is completely integrable.
\item Jacobi multipliers from normalizing vector fields: Let $g_1,\ldots,g_{n-1}$ be analytic vector fields on $U^*\subseteq U$ such that
\[
\left[g_i,\,f\right] =\lambda_i\cdot f,\quad 1\leq i\leq n-1
\]
with analytic functions $\lambda_i$. Then
\[
\psi(x):=\det(f(x),g_1(x),\ldots,g_{n-1}(x)),
\]
assuming that $\psi\not=0$, is an inverse Jacobi multiplier of system \eqref{deq}.
\end{enumerate}
}
\end{remark}

%{
The following result on PDNF, although quite straightforward, does not seem to be available in the literature.
\begin{proposition}\label{divnofoprop} Let $f=A_s+\cdots$ be in PDNF. Then:
\begin{enumerate}[$(a)$]
\item The divergence ${\rm div}\, f$ is a first integral of $A_s$.
\item If $\psi$ is an inverse Jacobi multiplier of $f$, then $X_{A_s}(\psi)={\rm div}\,A_s\cdot\psi$.
\end{enumerate}
\end{proposition}
\begin{proof}
To prove part $(a)$, we first assume that $A_s$ is in diagonal form \eqref{diagmat}. Then it is sufficient to prove part $(a)$ for every vector monomial $x_1^{m_1}\cdots x_n^{m_n}e_j$ that commutes with $A_s$, thus satisfies the resonance condition \eqref{resonance}. The divergence of this vector field is equal to
\[
\frac{\partial}{\partial x_j}\left(x_1^{m_1}\cdots x_n^{m_n}\right)=m_j\cdot x_1^{m_1}\cdots x_j^{m_j-1}\cdots x_n^{m_n}.
\]
In case $m_j=0$ the asssertion is obvious. In case $m_j>0$ the resonance condition may be rewritten as
\[
m_1\lambda_1+\cdots+(m_j-1)\lambda_j+\cdots+m_n\lambda_n=0,
\]
equivalently $x_1^{m_1}\cdots x_j^{m_j-1}\cdots x_n^{m_n}$ is a first integral of $A_s$.\\
In general, a vector field in PDNF will have the form $f^*(x)=T^{-1}f(Tx)$ with an invertible matrix $T$. This implies the identities
\[
Df^*(x)=T^{-1}Df(Tx)T \text{  and  }{\rm tr}\,Df^*(x)={\rm tr}\,Df(Tx)
\]
by conjugacy of the the Jacobian matrices. Now the semisimple part of $T^{-1}AT$ is equal to  $T^{-1}A_sT$, and one finds with $\theta(x):={\rm div}\,f(x)$
\[
X_{T^{-1}A_sT}(\theta)(Tx)=D\theta(Tx)TT^{-1}A_sTx=X_{A_s}(\theta)(Tx)=0,
\]
hence the divergence of $f^*$ is a first integral for $T^{-1}A_sT$, as asserted.
\\
Part $(b)$ follows with the second part in the proof of \cite{Wa00}, Lemma 2.2.
\end{proof}
\begin{remark}{\em
Part $(b)$ of the Proposition states that $\psi$ is an element of the $I(A_s)$-module $I_{{\rm div}\,A_s}(A_s)$, which is finitely generated by Lemma \ref{torfinlem}. In particular, whenever ${\rm div}\,A_s=0$ then an inverse Jacobi multiplier corresponds to a semi-invariant of the system reduced by invariants of $A_s$.
}
\end{remark}
%%%%%%%%%%%%%%%%%%%%%%%%%%%%%%%%%%%%%%%%%%%%%%%
\subsection{Examples}
We first record the simplest cases, for the sake of completeness. For the proof of uniqueness, note part $(b)$ of Remark \ref{jacobibasics}.
\begin{proposition}
Let $A={\rm diag}(\lambda_1,\ldots,\lambda_n)$ with the $\lambda_i$ linearly independent over $\mathbb Q$. Then (up to scalar multiples) the only inverse Jacobi multiplier of $\dot x=Ax$ $($the general PDNF with linear part $A$$)$ is $\sigma=x_1\cdots x_n$.
\end{proposition}

Turning to a more interesting setting, from Remark \ref{jacobibasics} and Lemma \ref{estprep}$(a)$ we get immediately:
\begin{proposition} Given the setting in Proposition \ref{singleprop} {with $k$ minimal} such that { $X_{U(_k}(\psi)\not=0$}, the vector field $f$ admits an inverse Jacobi multiplier of the form
\[
{\det(f(x), C_1x,\ldots,C_{n-1}x) =x_1\cdots x_n\ \psi(x)^k\cdot(c + {\text{\rm  h.o.t.}}).}
\]
with a nonzero constant $c$.
\end{proposition}
In particular this result takes care of three dimensional vector fields $f(x)=Ax+\cdots$ when the eigenvalues of $A$ span a two dimensional vector space over $\mathbb Q$.\\

We next discuss in some detail the three dimensional vector field
\begin{equation}\label{ifaceq}
f(x)=\begin{pmatrix} 1&0&0\\
0&-1&0\\
0&0&0\end{pmatrix}\cdot x+\begin{pmatrix}\alpha_1 x_1x_3\\
                                                                          \alpha_2 x_2x_3\\
                                                                         \alpha_3 x_3^2+\alpha_4 x_1x_2\end{pmatrix} +\cdots=Ax+\sum f_j(x)
\end{equation}
in Poincar\'e-Dulac normal form. Note that the homogeneous term of degree two is of the most general form, with
\[
{\rm div}\,f(x)=(\alpha_1+\alpha_2+2\alpha_3)\,x_3+\cdots.
\]
This represents a case in dimension three with the eigenvalues spanning a one dimensional vector space over $\mathbb Q$.
We will see by this example that Jacobi multipliers may exist even if they cannot be constructed with centralizer elements according to Remark \ref{jacobibasics}$(c)$.  Moreover we show that generically this vector field does not admit a formal inverse Jacobi multiplier.
\begin{proposition}\label{ifacprop}
Let $f$ be given by \eqref{ifaceq}, and assume that $\alpha_1,\,\alpha_2,\,\alpha_3,\,\alpha_4$ are linearly independent over the rational number field $\mathbb Q$, and $\alpha_3=1$ with no loss of generality $($by scaling $x\mapsto \rho\cdot x$$)$. Then:
\begin{enumerate}[$(a)$]
\item The dimension of $\mathcal C^{\rm for}(f)$ is equal to 2.
\item The quadratic system $f=A+f_2$ admits a unique inverse Jacobi multiplier $($up to scalar multiples$)$
\[
\psi(x)=x_1x_2x_3^2+\beta^*x_1^2x_2^2,\quad\text{with  } \beta^*:=\frac{2\alpha_4}{2-\alpha_1-\alpha_2}.
\]
\item For a Zariski open and dense subset in the coefficient space for $f_3$, the vector field $f$ does not admit an inverse Jacobi multiplier.
\end{enumerate}
\end{proposition}
\begin{remark} {\em The assumption on the linear independence of the $\alpha_i$ over $\mathbb Q$ is satisfied for all $\alpha=(\alpha_1,\ldots,\alpha_4)^\tau$ in the complement of a measure zero subset of $\mathbb C^4$. In this sense the assumption on the $\alpha_i$ in the Proposition is generic.
}
\end{remark}
\begin{proof}[Proof of Proposition \ref{ifacprop}] Part $(a)$ follows from \cite{Wa00b}, Example 4. The proof of part $(b)$ requires several steps.
\begin{enumerate}[$(i)$]
\item We assume the existence of an inverse Jacobi multiplier
\[
\phi=\phi_r+\cdots,\quad \phi_r\not=0.
\]
Evaluating the condition $X_f(\phi)={\rm div}\,f\cdot\phi$ for the lowest degrees, with ${\rm div}\,A=0$ one gets
\begin{align*}
X_A(\phi_r)&=0;\\
X_{f_2}(\phi_r)+X_A(\phi_{r+1})&=(\alpha_1+\alpha_2+2)\,x_3\phi_r;\\
X_{f_3}(\phi_r)+X_{f_2}(\phi_{r+1})+X_A(\phi_{r+2})&=(\alpha_1+\alpha_2+2)x_3\,\phi_{r+1}+ {\rm div}\,f_3\,\phi_r.
\end{align*}
From Proposition \ref{divnofoprop} we know that all $X_A(\phi_{r+j})=0$, and there remains
\begin{align}
X_{f_2}(\phi_r)&=(\alpha_1+\alpha_2+2)\,x_3\phi_r;\label{ifcond2}\\
X_{f_3}(\phi_r)+X_{f_2}(\phi_{r+1})&=(\alpha_1+\alpha_2+2)x_3\,\phi_{r+1}+ {\rm div}\,f_3\,\phi_r.\label{ifcond3}
\end{align}
\item We evaluate equation \eqref{ifcond2}. For $c:=e_3$ we have $f_2(c)=c$ and $Df_2(c)={\rm diag}\,(\alpha_1,\alpha_2,2)$, thus Lemma \ref{adelem} from the Appendix is applicable, and yields
\begin{align*}
k_1\alpha_1+k_2\alpha_2+2k_3+k&=\alpha_1+\alpha_2+2;\\
k_1+k_2+k_3+k&=r.
\end{align*}
Since the $\alpha_i$ are linearly independent over the rationals, we get $k_1=k_2=1$ and $k+2k_3=2$. This leaves only two possibilities:
\[
r=3\text{  and  }k=0, \text{  or  } r=4\text{  and  } k=2.
\]
\item The algebra of polynomial first integrals of $\dot x=Ax$ is generated by $\psi_1:=x_1x_2$ and $\psi_2:=x_3$ (see Lemma \ref{fireslem} and the subsequent Remark), so $\phi_r$ is a polynomial in $\psi_1$ and $\psi_2$.
\begin{itemize}
\item Assume that $r=3$ and $k=0$, then we have
\[
\phi_3=\beta_1x_3^3+\beta_2x_1x_2x_3
\]
with constants $\beta_i$, and $D\phi_3(x)c=0$ due to $k=0$. But this implies that $3\beta_1x_3^2+\beta_2x_1x_2=0$, hence $\beta_1=\beta_2=0$ and the contradiction $\phi_3=0$. We conclude that this case cannot occur.
\item In case $r=4$ we have
\[
\phi_4=\beta_1x_3^4+\beta_2x_1x_2x_3^2+\beta_3x_1^2x_2^2
\]
with constants $\beta_i$. From $k=2$ we obtain\footnote{{See Lemma \ref{adelem} for the notation used in the following.}}
\[
0=D^3\phi_4(x)(c,c,c)=24\beta_1x_3,
\]
 hence $\beta_1=0$, and then $0\not=D^2\phi(x)(c,c)=2\beta_2x_1x_2$ shows that necessarily $\beta_2\not=0$.

\end{itemize}
\item Thus, up to a nonzero scalar factor,
\[
\phi_4=x_1x_2x_3^2+\beta x_1^2x_2^2.
\]
Evaluating the condition
\[
X_{f_2}(\phi_4)=(\alpha_1+\alpha_2+2)\,x_3\phi_4
\]
for degree two by straightforward computations, one finds that it is satisfied if and only if $\beta=\beta^*$. The proof of part $(b)$ is complete.
\item We now consider $f$ with the special cubic term
\[
f_3(x)=\begin{pmatrix} 0\\ 0\\ x_3^3\end{pmatrix},
\]
make the ansatz
\[
\phi_5=\theta_1x_3^5+\theta_2x_1x_2x_3^3+\theta_3x_1^2x_2^2x_3
\]
and evaluate equation \eqref{ifcond3}. Straightforward computation and comparison of coefficients yields
\begin{align*}
\text{for  } x_1^2x_2^2x_3^2:&\quad (\alpha_1+\alpha_2-1)\theta_3+3\alpha_4\theta_2=3\beta^*;\\
\text{for  } x_1x_2x_3^4:&\quad 5\alpha_4\theta_1+\theta_2=1;\\
\text{for  } x_1^2x_2^2x_3^2:&\quad 5\theta_1=(\alpha_1+\alpha_2+2)\theta_1;\\
\text{for  } x_1^2x_2^2x_3^2:&\quad \alpha_4\theta_3=0.
\end{align*}
The last two equations directly show that $\theta_1=\theta_3=0$ due to the linear indepencdence of the $\alpha_i$, and there remains $\theta_2=1$ from the second equation, which leads to the contradiction $\alpha_4=\beta^*$ in the first equation.
\item We have shown that equation \eqref{ifcond3} has no solution when the cubic term has the special form $(0,0,x_3^3)^\tau$. Generally \eqref{ifcond3} may be viewed as an inhomogeneous linear system of equations for the coefficients of $\phi_5$, with the coefficients of $f_3$ as parameters. Since the system has no solution for a special choice of $f_3$, it has no solution for all $f_3$ with coefficients in a Zariski-open and dense subset of coefficient space.
\end{enumerate}
\end{proof}
%%%%%%%%%%%%%%%%%%%%%%%%%%%%%%%%%%%%%%%%%%%%%%%%
\subsection{A distinguished class of examples.}
In this subsection we consider vector fields $f=A+\cdots$ in PDNF with the linear part satisfying the following property:
\begin{equation}\label{onedivgen}
\text{ The } I(A_s)\text{-module } I_{{\rm div}\,A_s}(A_s)\text{ is generated by }\sigma:=x_1\cdots x_n.
\end{equation}
We note that ${\rm div}\,A_s\not=0$ in this case, otherwise the module would be generated by $1$. Moreover:
\begin{lemma}\label{onedivlem}
Given condition \eqref{onedivgen}, the module ${\mathcal C}^{\rm for}(A_s)$ is generated by $Q_1,\ldots,Q_n$, with $Q_i(x)=x_ie_i$. In particular $A=A_s$.
\end{lemma}
\begin{proof}
Assume w.l.o.g.\  that
\[
\lambda_1=\sum_{i=2}^n m_i\lambda_i
\]
with nonnegative integers $m_i$. Then
\[
{\rm div}\,A_s=\sum_{j=1}^n\lambda_j=\sum_{i=2}^n (m_i+1) \lambda_i,
\]
thus both $\sigma$ and $x_2^{m_2+1}\cdots x_n^{m_n+1}$ are elements of ${\mathcal C}^{\rm for}(A_s)$, hence $\sigma$ cannot generate this module.
\end{proof}
We obtain rather definitive results for vector fields in this class if they satisfy a further condition on the algebra of first integrals.
\begin{theorem}\label{disalgjacthm}
Let $f=A+\cdots$ in PDNF satisfy condition \eqref{onedivgen} and assume moreover that $I(A_s)$ admits the algebraically independent generator system of monomials $\phi_1,\ldots,\phi_r$, satisfying
\[
X_{Q_j}(\phi_i)=m_{ij}\phi_i\text{  with nonnegative integers  }m_{ij},\quad 1\leq i\leq r,\,1\leq j\leq n.
\]
\begin{enumerate}[$(a)$]
\item Then one has
\[
f=A+\sum \eta_jQ_j,
\]
with
\[
\eta_i(x)=\widehat\eta_i(\phi_1(x),\ldots,\phi_r(x)),\quad 1\leq i\leq r,
\]
and the reduced vector field has then the form
\[
\widehat f(y)=\begin{pmatrix} y_1\,\sum m_{1j}\widehat \eta_j\\
                                                                   \vdots \\
                                                 y_r\,\sum m_{rj}\widehat \eta_j\end{pmatrix}.
\]
\item If $\psi$ is an inverse Jacobi multiplier for $f$, thus $\psi=\sigma\cdot\rho$ with
\[
\rho(x)=\widehat\rho(\phi_1(x),\ldots,\phi_r(x)),\quad \widehat\rho\in\mathbb C[[y_1,\ldots,y_r]],
\]
then
\[
\widetilde \psi:=y_1\cdots y_r\cdot\widehat\rho
\]
is an inverse Jacobi multiplier for the reduced vector field $\widehat f$.
\item Let the quadratic part of $\widehat f$ be given as
\[
\widehat f_2=\begin{pmatrix} y_1\,\sum \nu_{1j}y_j\\
                                                                   \vdots \\
                                                 y_r\,\sum \nu_{rj}y_j\end{pmatrix} \text{  with }\nu_{ij}\in\mathbb C.
\]
Then
\begin{itemize}
\item Whenever $r\geq 3$ then $\widehat f_2$ admits no inverse Jacobi multiplier for $(\nu_{ij})$ in a subset of $\mathbb C^{r\times r}$ of full measure. Consequently the reduced vector field admits no inverse Jacobi multiplier whenever $\widehat f_2$ is of this form.
\item In case $r=2$ there exists a subset of $\mathbb C^{2\times 2}$ of full measure such that $\widehat f_2$ with coefficients in this subset admits a unique inverse Jacobi multiplier $($up to scalar multiples$)$. Moreover for every $(\nu_{ij}^*)$ in this set there is a subset of full measure in the coefficient space of $\widehat f_3$ such that $\widehat f_2+\widehat f_3+\cdots$ does not admit an inverse Jacobi multiplier for the coefficients of $\widehat f_3$ in this latter set.
\end{itemize}
\end{enumerate}
\end{theorem}
\begin{remark}{\em
Note that the condition that the generators be monomials poses no restriction.
}
\end{remark}
\begin{proof}[Proof of Theorem \ref{disalgjacthm}]
For the statements in part $(a)$ compare Theorem \ref{disalgcentthm} and its proof. We now address the statements in parts $(b)$ and $(c)$.
\begin{enumerate}[$(i)$]
\item The first statement of part $(b)$ is a consequence of condition \eqref{onedivgen}. The condition for an inverse Jacobi multiplier reads as
\[
X_f(\sigma)\rho+\sigma X_f(\rho)=X_f(\sigma \rho)={\rm div}\,f\cdot\sigma\rho,
\]
and we compute
\[
X_f(\sigma)=({\rm div}\,A+\sum \eta_i)\,\sigma,
\]
\[
{\rm div}\,f={\rm div}\,A+\sum\left(\eta_i+x_i\frac{\partial \eta_i}{\partial x_i}\right).
\]
Thus the defining condition holds if and only if
\[
\sum\eta_ix_i\frac{\partial \rho}{\partial x_i}= X_f(\rho)=\sum x_i\frac{\partial \eta_i}{\partial x_i}\cdot\rho.
\]
Now we have
\[
\begin{split}%{rcl}
x_i\frac{\partial \rho}{\partial x_i}&=\sum_j\frac{\partial \widehat \rho}{\partial y_j}(\phi_1,\ldots,\phi_r)\cdot x_i\frac{\partial \phi_j}{\partial x_i}\\
                                                      &=\sum_j m_{ji}\phi_j\frac{\partial \widehat \rho}{\partial y_j}(\phi_1,\ldots,\phi_r)
\end{split}
\]
and
\[
\sum_i\eta_ix_i\frac{\partial \rho}{\partial x_i}=\sum_{i,j} m_{ji}\eta_i\phi_j\frac{\partial \widehat \rho}{\partial y_j}(\phi_1,\ldots,\phi_r)
\]
Likewise we evaluate the right hand side of the defining condition to obtain
\[
\sum x_i\frac{\partial \eta_i}{\partial x_i}\cdot\rho=\sum_{i,j}m_{ji}\phi_j\frac{\partial\widehat\eta_i}{\partial y_j}(\phi_1,\ldots,\phi_r)\cdot\rho.
\]
\item
Due to the algebraic independence of the $\phi_j$, the defining condition for an inverse Jacobi multiplier $\sigma\cdot\rho$ is equivalent to the condition
\[ %begin{equation}\label{jaccond}
\sum_{i,j} m_{ji}\widehat\eta_iy_j\frac{\partial \widehat \rho}{\partial y_j}=\sum_{i,j}m_{ji}y_j\frac{\partial\widehat\eta_i}{\partial y_j}\cdot\widehat\rho
\] %end{equation}
in $\mathbb C[[y_1,\ldots,y_r]]$.
Now a straightforward computation shows that this condition holds if and only if $\widetilde\psi$ is an inverse Jacobi multiplier for the reduced vector field. Thus part $(b)$ is proven.
\item In the following we use the fact that the lowest order term of an inverse Jacobi multiplier of $\widehat f$ is an inverse Jacobi multiplier of the lowest order term $\widehat f_2$. In particular, if $\widehat f_2$ admits no inverse Jacobi multiplier then neither does $\widehat f$.
\item In order to establish degree bounds for inverse Jacobi multipliers of $\widehat f_2$ in a generic case, we use some notation and take some arguments from the proofs of Theorem \ref{quadcentthm}, part $(iii)$ and Theorem \ref{disalgcentthm}, part $(iv)$. Thus we assume $\nu_{11}\not=0$, hence $c:=\nu_{11}^{-1}e_1$ satisfies $\widehat f_2(c)=c$, and $D\widehat f_2(c)$ has eigenvalues
\[
2,\,\nu_{21}\nu_{11}^{-1},\,\ldots,\nu_{r1}\nu_{11}^{-1}.
\]
Moreover
\[
{\rm div}\,\widehat f_2=\sum_{i,j}\nu_{ij}y_j+\sum_i\nu_{ii}y_i\Rightarrow {\rm div}\,\widehat f_2(c)=2+\nu_{21}\nu_{11}^{-1}+\cdots+\nu_{r1}\nu_{11}^{-1}.
\]
Now let the degree of the inverse Jacobi multiplier $\widetilde\psi$ be equal to $s$. By Lemma \ref{adelem} in the Appendix there exist nonnegative integers $k,\,k_1,\,\ldots,k_r$ such that
\[
\begin{array}{rcl}
k+k_1+\cdots+k_r&=&s,\\
k+2k_1+\nu_{21}\nu_{11}^{-1}k_2+\cdots+\nu_{r1}\nu_{11}^{-1}k_r&=&2+\nu_{21}\nu_{11}^{-1}+\cdots+\nu_{r1}\nu_{11}^{-1}.
\end{array}
\]
If $\nu_{11},\ldots,\nu_{r1}$ are linearly independent over the rational numbers then the second identity forces
\[
k_2=\cdots=k_r=1\text{ and }k+2k_1=2.
\]
Thus either $k=0$ and $k_1=1$, or $k=2$ and $k_1=0$. In total we have the alternative
\[
s=r, \text{  or  } s=r+1.
\]
\item Moreover $\widetilde\psi=y_1\cdots y_r\cdot\widehat\rho$, hence $\widehat\rho$ is constant or linear, and the inverse Jacobi multiplier condition is equivalent to
\[
\sum_{i,j}\nu_{ij}\frac{\partial\widehat\rho}{\partial y_i} y_iy_j=\sum_i\nu_{ii}\widehat\rho y_i.
\]
Since $\nu_{11}\not=0$, $\widehat\rho$ cannot be constant, thus must have degree one. Setting
\[
\widehat\rho(y)=\sum_i\alpha_iy_i
\]
the inverse Jacobi multiplier condition is equivalent to
\[
\sum_{i,j}\alpha_i\nu_{ij}y_iy_j=\sum_{i,j}\alpha_j\nu_{ii}y_iy_j.
\]
Compare coefficients of the monomials involved: For $i=j$ the coefficients are automatically equal, for $i<j$ one finds
\[
\mu_{ij}\alpha_i+\mu_{ji}\alpha_j=0, \text{  with  } \mu_{ij}:=\nu_{ij}-\nu_{jj}.
\]
Whenever $r\geq 3$ there is a Zariski open set in $U\subseteq\mathbb C^{r\times r}$ such that this system of $r(r-1)/2$ equations for the $\alpha_i$ has only the trivial solution: If the $\mu_{j1}$ and $\mu_{1j}$, $2\leq j\leq r$ are all nonzero then $\alpha_2,\ldots,\alpha_r$ are uniquely determined as multiples of $\alpha_1$, and nonzero whenever $\alpha_1\not=0$. The further equation $\mu_{23}\alpha_2+\mu_{32}\alpha_3=0$ now  imposes nontrivial conditions on $\mu_{32}$ and $\mu_{23}$ whenever $\alpha_1\not=0$, which imply a nontrivial polynomial relation between $\mu_{j1},\,\mu_{1j},\,\mu_{23},\,\mu_{32}$. We have thus shown that all $\alpha_i$ must be zero on a nonempty Zariski open set $U$. Moreover, whenever $\nu_{11},\ldots,\nu_{r1}$ are linearly independent over $\mathbb Q$ then there exist $\nu_{ij}$ with $j>1$ such that the corresponding $(\mu_{ij})\in U$. We have thus shown that in case $r\geq 3$ no inverse Jacobi multiplier generically exists for $\widehat f_2$, hence for $\widehat f$.
\item There remains the second assertion in part $(c)$. The argument above for the case $r=2$ yields one (generically nontrivial) linear equation for $\alpha_1,\,\alpha_2$ and thus uniqueness of the inverse Jacobi multiplier for $\widehat f_2$. The remaining assertion was proven in \cite{Wa03}, Thm. 2.10.
\end{enumerate}
\end{proof}
\begin{corollary}
Let $h(x)=Bx+\sum_{j\geq 2} h_j(x)$ be a formal vector field on $\mathbb C^n$, $n\geq 3$, such that $B=B_s$ satisfies condition \eqref{onedivgen} and furthermore $I(B)$ admits an algebraically independent generator system. Then there exists an $L>0$ with the following property: The coefficient space $\widetilde W$ of $h_2+\cdots+h_L$ contains a full measure subset $\widetilde {\mathcal R}$ such that for coefficients in $\widetilde {\mathcal R}$ there exists no formal inverse Jacobi multiplier of $h$.
\end{corollary}

For dimension three the above results contribute to a relatively complete picture:
\begin{example}{\em
Let $\ell_i>1$, $1\leq i\leq 2$ be relatively prime, moreover $d_1^*,\,d_2^*\in\mathbb Z_+$ with ${\rm gcd}\,(d_1^*,\,\ell_1)={\rm gcd}\,(d_2^*,\,\ell_2)=1$, and $A={\rm diag}\,(\ell_2d_1^*,\,\ell_1 d_2^*,\,-\ell_1\ell_2)$. (According to Proposition \ref{dim3prop} such vector fields are exactly those which correspond to the distinguished algebraic setting discussed in subsection \ref{subsec42}.) Then $A=A_s$ also satisfies condition \eqref{onedivgen}, as one sees by this argument: A relation
\[
m_1\ell_2d_1^*+m_2\ell_1d_2^*-m_3\ell_1\ell_2=\ell_2d_1^*+\ell_1d_2^*-\ell_1\ell_2
\]
with $m_i\in\mathbb Z_+$ implies that $\ell_1$ divides $(m_1-1)\ell_2d_1^*$, hence $\ell_1|(m_1-1)$ by relative primeness, and consequently $m_1>0$. By the same token one finds $m_2>0$, and finally $m_3=0$ is impossible since $-\ell_1\ell_2<0$. Now Theorem \ref{disalgjacthm} shows that a vector field $f=A+\cdots$ generically admits no formal inverse Jacobi multiplier.
}
\end{example}

%%%%%%%%%%%%%%%%%%%%%%%%%%%%%%%%%%%%%%%%%%%%%%%%%%
\section{Appendix}
\subsection{Some linear algebra}
For easy reference we collect here some facts that are used in the main part of the paper.
\begin{lemma}\label{lincommutelem}
\begin{enumerate}[$(a)$]
\item Let $B={\rm diag}\,\left(\mu_1,\ldots,\mu_n\right)$, with the equal ones among the $\mu_i$ listed consecutively, thus
\[ \mu_1=\cdots=\mu_{s_1}, \,\mu_{s_1+1}=\cdots=\mu_{s_1+s_2},\ldots,\mu_{s_1+\cdots+s_r+1}=\cdots=\mu_n,
\]
and the elements of the blocks pairwise different. Then $D$ commutes with $B$ if and only if
\[
D=\begin{pmatrix} D_1&0&\cdots&0\\
                                0&D_2& & \vdots\\
                                 \vdots& &\ddots &0\\
                                 0&\cdots&0& D_{r+1}\end{pmatrix}
\]
is in block diagonal form, with blocks of appropriate sizes.
\item  Let $N\in\mathbb C^{\ell\times\ell}$ be a strict upper triangular matrix, of rank $\ell-1$. Then $C$ commutes with $N$ if and only if $C=\sum_{i=0}^{\ell-1} \alpha_i N^i$, with $\alpha_i\in \mathbb C$.
In particular the space of matrices commuting with $N$ has dimension $\ell$.
\item In the subspace of all strictly upper triangular $\ell\times\ell$ matrices, the matrices of rank $\ell-1$ form a Zariski open and dense subset.
\end{enumerate}
\end{lemma}
\begin{proof}
Part $(a)$ is standard knowledge, and can be verified by direct calculations. For part $(b)$ we use that $N$ is conjugate to a Jordan block
\[
D=\begin{pmatrix} 0&1&0&\cdots&0\\
                             \vdots&\ddots&\ddots &\ddots &\vdots\\
                             \vdots& &\ddots &\ddots &0\\
                               \vdots& & &\ddots &1\\
                               0&\cdots&  \cdots&\cdots &0
                              \end{pmatrix}
\]
by an upper triangular matrix, and for this Jordan block the assertions may be verified by direct calculations. To prove part $(c)$, note that the determinant of the $(\ell-1)\times(\ell-1)$ submatrix with the first column and the last row deleted is generically nonzero.
\end{proof}
%%%%%%%%%%%%%%%%%%%%%%%%%%%%%%%%%%%%%%%%%%%%%%%%%%
\subsection{Some results on quadratic vector fields}
First we recall and restate \cite{WaADE}, Proposition 4.7; the language used below is more convenient for the purpose of the present paper. {As a matter of notation, we consider higher derivatives as multilinear maps. Thus, for analytic  $f:\,\mathbb C^n\to \mathbb C^m$ the $k^{\rm th}$ derivative $D^k\,f(x)$ at $x$ is recursively defined by
\[
D^{k+1}\,f(x)(y_1,\ldots,y_k,y_{k+1}):=D\left(D^k\,f(x)(y_1,\ldots,y_k)\right)y_{k+1}
\]
for $y_1,\ldots,y_{k+1}\in\mathbb C^n$. Note that this multilinear map is also symmetric in the $y_i$, due to equality of mixed partial derivatives.
}
\begin{lemma}\label{adelem}
Let $p$ be a homogeneous quadratic vector field on $\mathbb C^n$, and $\gamma$ a $($nonzero$)$ homogeneous semi-invariant of degree $s>0$ for this vector field, thus one has an identity
\[
D\gamma(x)p(x)=\lambda(x)\gamma(x)
\]
with some linear form $\lambda$, the cofactor of $\gamma$.
 Moreover assume there exists $c\in \mathbb C^n$ such that $p(c)=c\not=0$. Then the following hold:
\begin{enumerate}[$(a)$]
\item If $\gamma(c)\not=0$ then $\lambda(c)=s$.
\item In case $\gamma(c)=0$ let $k$ be such that $0\leq k<s$ and
\[
D^k\gamma(x)(c,\ldots,c)\not=0,\quad D^{k+1}\gamma(x)(c,\ldots,c)=0.
\]
Then $x\mapsto D^k\gamma(x)(c,\ldots,c)$ is a semi-invariant of degree $s-k$ for the linear map $B: \mathbb C^n\to\mathbb C^n$, $Bx=Dp(c)x$, with cofactor $\lambda(c)-k$. If $\alpha_1,\ldots,\alpha_n$ are the eigenvalues of $B$ $($each counted according to multiplicity$)$ then there exist nonnegative integers $k_1,\ldots,k_n$ such that
\[
\sum k_i=s-k\text{  and  } \lambda(c)=k+\sum k_i\alpha_i.
\]
\end{enumerate}
\end{lemma}
{
Next we recall and specialize \cite{WaADE}, Lemma 10.4.
\begin{lemma}\label{adelemcent} Let $p$ be a homogeneous quadratic vector field on $\mathbb C^r$, and $q\not=0$ a homogeneous vector field of degree $s>1$ such that $[p\,,q]=0$.
Moreover let $c$ be such that $p(c)=c\not=0$, with eigenvalues $\mu_1,\ldots,\mu_r$ of $Dp(c)$. $($Since $2$ is an eigenvalue with eigenvector $c$ by homogeneity, we set $\mu_1=2$.$)$ Then there exist nonnegative integers $\ell,\ell_1,\ldots,\ell_r$ and some $k\in\{1,\ldots,r\}$ such that
\begin{equation}\label{quadraticcent}
\sum\ell_i+\ell=s\text{ and }\sum\ell_i\,\mu_i+\ell=\mu_k.
\end{equation}
\end{lemma}
}
\medskip

\noindent{{\bf Acknowledgements.}
 NK gratefully acknowledges support by the DFG Research Training Group GRK 1632 ``Experimental and Constructive Algebra''.  SW thanks the School of Mathematical Sciences of Shanghai Jiao Tong University for its hospitality and for creating a congenial research environment during a visit in fall 2019.
} XZ is partially supported by NNSF of China grant numbers 11671254 and 11871334
%%%%%%%%%%%%%%%%%%%%%%%%%%%%%%%%%%%%%%%%%%%%%%%

%}

\end{document}